\documentclass[12pt] {amsart}

\usepackage {amssymb}
\usepackage {amsmath}
\usepackage {amsthm}
\usepackage{graphicx}
\usepackage {amscd}
\usepackage[colorlinks, citecolor=blue]{hyperref}
\input{diagrams.sty} 

\theoremstyle{plain}
\newtheorem{thm}{Theorem}[section]
\newtheorem{theorem}[thm]{Theorem}
\newtheorem{lemma}[thm]{Lemma}
\newtheorem{lem}[thm]{Lemma}

\newtheorem{cor}[thm]{Corollary}
\newtheorem{corollary}[thm]{Corollary}
\newtheorem{prop}[thm]{Proposition}

\theoremstyle{definition}

\newtheorem{rmk}[thm]{Remark}

\newtheorem{defn}[thm]{Definition}

\newtheorem{claim}[thm]{Claim}

\newcommand{\bF}{{\mathbb F}}

\newcommand{\N}{{\mathbb N}}

\newcommand{\Q}{{\mathbb Q}}

\newcommand{\Mob}{{\mathbf{Mob} }}
\author{Christopher D. Hacon}
\address{Department of Mathematics, University of Utah, 155 South 1400 East, Salt Lake City, UT 48112-0090, USA} \email{hacon@math.utah.edu}
\author{Chenyang Xu}
\address{Beijing International Center of Mathematics Research, 5 Yiheyuan Road, Haidian District, Beijing, 100871, China}\email{cyxu@math.pku.edu.cn}

\begin{document}
\title[On the three dimensional MMP in ${\rm char}(p)>0$]{On the three dimensional minimal model program in positive characteristic}
\begin{abstract} 
Let $f:(X,B)\to Z$ be a $3$-fold extremal dlt flipping contraction defined over an algebraically closed field of characteristic $p>5$, such that the coefficients of $\{ B\}$ are in the standard set $\{ 1-\frac 1n|n\in \mathbb N\}$, then the flip of $f$ exists. As a consequence, we prove the existence of minimal models for any projective $\Q$-factorial terminal variety $X$ with pseudo-effective canonical divisor $K_X$.
\end{abstract}

\maketitle
\tableofcontents
\section{Introduction}
The minimal model program (MMP) is one of the main tools in the classification of higher dimensional algebraic varieties. It aims to generalize to dimension $\geq 3$ the results obtained by 
the Italian school of algebraic geometry at the beginning of the 20-th century.

In characteristic  $0$, much progress has been made towards establishing the minimal model program.
In particular the minimal model program is true in dimension $\leq 3$, and in higher dimensions, it is known that the canonical ring is finitely generated, flips and divisorial contractions exist, minimal models exist for varieties of general type, and we have termination of flips for the minimal model program with scaling on varieties of general type (see \cite{BCHM10} and the references contained therein).
The fundamental tool used in establishing these results is Nadel-Kawamata-Viehweg vanishing (a powerful generalization of Kodaira vanishing).

Unluckily, vanishing theorems are known to fail for varieties in characteristic $p>0$ and so very little is known about the minimal model program in characteristic $p>0$. Another serious difficulty is that resolution of singularities
is not yet known in characteristic $p>0$ and dimension $>3$.
The situation is as follows:
in dimension $2$, the full minimal model program holds (see \cite{KK, Tanaka12a} and references therein).
In dimension $3$, resolution of singularities is known (see \cite{Abhyankar98,Cutkosky04,CP08,CP09}). Partial results towards the existence of divisorial and flipping contractions are proven in \cite{Keel99}.
Termination of flips for terminal pairs, holds by the usual counting argument and Kawamata has shown that the existence of relative minimal models for semistable families when $p>3$ \cite{Kawamata94}.
Thus the main remaining questions are the base point free theorem, the existence of flips and abundance.
In this paper we prove the following. 

\begin{theorem}\label{t-flipexists}
Let $f:(X,B)\to Z$ be an extremal flipping contraction of a dlt threefold defined over an algebraically closed field of characteristic $p>5$ such that the coefficients of $\{ B\}$ belong to the standard set $\{ 1-\frac 1 n|n\in \mathbb N\}$. Then the flip exists. 
\end{theorem}

We have the following result on the existence of minimal models.
\begin{thm} \label{t-MM}
Let $(X,\Delta)$ be a $\Q$-factorial projective three dimensional  canonical pair over an algebraically closed field $k$ of characteristic  $p>5$. Assume all coefficients of $ \Delta $ are in the standard set $\{ 1-\frac 1 n | n\in \mathbb N\}$ and $N_\sigma (K_X+\Delta)\wedge \Delta = 0$. If $K_X+\Delta$ is pseudo-effective, then
 \begin{enumerate}
 \item  there exists a minimal model $X^{m}$ of $(X,\Delta)$, and
\item if, moreover,  $k=\overline{\bF}_p$, then $X^m$ can be obtained by running the usual $(K_X+\Delta)$-MMP. 
\end{enumerate}
\end{thm}
We remark that in general  the minimal model in (1) is not obtained by running the MMP in the usual sense unless $k=\overline{\bF}_p$. See Section \ref{s-mmp} for more details.

\subsection {Sketch of the proof}
After Shokurov's work, it has been known in characteristic 0 that the existence of flips can be reduced to a special case, called {\it pl flips} (see \cite{Shokurov92,Fujino07}). We apply the same idea in characteristic $p$. The key result of this paper is the proof of the existence of pl-flips cf. \eqref{t-2}. 
Since the base point free theorem has not yet been established in full generality for threefolds in characteristic $p>0$, instead of running the MMP in the usual sense, we run a variant of the MMP, which yields a minimal model (of terminal 3-folds, see Section \ref{s-mmp}). 

 Our strategy to show the existence of pl-flips will follow closely ideas of Shokurov as explained in \cite[Chapter 2]{Corti07}. It is well known that if $f:(X,S+B)\to Z$ is a pl-flipping contraction (so that $X$ is $\Q$-factorial, $(X,S+B)$ is plt and $-(K_X+S+B)$ and $-S$ are ample over $Z$) then the existence of the pl-flip is equivalent to the finite generation of the restricted algebra $R_{S/Z}(K_X+S+B)$. When $Z$ is affine, then the $m$-th graded piece of this algebra is given by the image of the restriction map $$H^0(X,\mathcal O _X(m(K_X+S+B)))\to H^0(S,\mathcal O_S(m(K_S+B_S)))$$ (where $S$ is normal and $B_S$ is defined by adjunction $(K_X+S+B)|_S=K_S+B_S$). For any proper birational morphism $g:Y\to X$ we can consider the corresponding plt pair $K_Y+S'+B_Y$ where $S'$ is the strict transform of $S$. The moving part of the image of the restriction of $H^0(Y,\mathcal O _Y(m(K_Y+S'+B_Y)))$ to $S'$ gives a mobile b-divisor $\mathbf M_{m,S'}$ corresponding to the $m$-th graded piece of $R_{S/Z}(K_X+S+B)$. Dividing by $m$, we obtain a sequence of non-decreasing $\mathbb{Q}$-b-divisor $\frac 1 m \mathbf M_{m,S'}$. The required finite generation is equivalent to showing that this non-decreasing sequence eventually stabilizes  (so that $\frac 1 m\mathbf M_{m,S'}$ is fixed for all $m>0$ sufficiently divisible and in particular these divisors descend to a fixed birational model of $S$).   

In characteristic 0, the proof has three steps: first, we show that this sequence of b-divisors descends to a fixed model (in fact they descend to $\bar S$ the terminalization of $(S,B_S)$); next we show that the limiting divisor is a $\mathbb{Q}$-divisor (instead of an arbitrary $\mathbb{R}$-divisor); finally we show that for sufficiently large degree, the sequence stabilizes.  All of these steps rely heavily on the use of vanishing theorems.

In characteristic $p>0$, the main difficulty is (of course) that the Kawamata-Viehweg vanishing theorem fails. However, after \cite{HH90}, techniques involving the Frobenius map have been developped to recover many of the results which are traditionally deduced from vanishing theorems (see Subsection \ref{F-sing}).  In this paper, we will use these results for lifting sections from the divisor $S$ to the ambient variety $X$. However, one of the difficulties we encounter is that we can only lift the sections in $S^0(\bullet)\subset H^0(\bullet)$, which is roughly speaking the subspace given by the images of the maps induced by iterations of the Frobenius. Thus it is necessary to understand under what conditions the inclusion $S^0(\bullet)\subset H^0(\bullet) $ is actually an equality. It is easy to see that if $(C,\Delta)$ is
a $1$-dimensional klt pair and $L$ is  a sufficiently ample line bundle, then $H^0(C,L)=S^0(C,\sigma(C,\Delta)\otimes L)$ (cf. \eqref{l-cu2}). Unluckily, in higher dimensions, this frequently fails. 

It turns out that the first two steps in the proof of characteristic 0,  depend only on being able to lift sections from curves (corresponding to general divisors in $|\mathbf M_{m,S'}|$). So after a suitable modification, the argument works in general (an added difficulty is that since Bertini's theorem fails in positive characteristic, these curves may be singular).  
Unluckily, the third step uses a lifting result in the surface case. This lifting result is subtler, however when the coefficients of $B$ are in the standard set $\{ 1-\frac 1 n |n\in \N \}$ and the characteristic is $p>5$, we are able to prove that $S^0(\bullet)=H^0(\bullet)$ even in the surface case (see \eqref{l-liftingsurface}). 

\noindent{\bf Acknowledgement:}  We are indebted to J\'anos Koll\'ar, James M$^{\rm c}$Kernan, Mircea Musta\c{t}\v{a}, Yuri Prokhorov, Yuchen Zhang and Runpu Zong for helpful conversations. We are especially grateful to Zsolt Patakfalvi, Karl Schwede and Hiromu Tanaka for a careful reading of our early drafts with many helpful suggestions. 

The first author was partially supported by NSF research grants no: DMS-0757897 and DMS-1300750 and a grant from the Simons foundation, the second author was partially supported by the Chinese government grant `Recruitment Program of Global Experts'.

\section{Preliminaries}
\subsection{Notation and Conventions}\label{NC} 
We work over an algebraically closed field $k$ of characteristic $p>0$. 
 We will use the standard notation  in \cite{KM98}.  In particular, see \cite[0.4, 2.34, 2.37]{KM98} for the definitions of
{\it terminal, klt, plt} and {\it dlt} singularities. 

We use the notation {\it b-divisor} as defined in  \cite[Subsection 2.3.2]{Corti07}. In particular, see [ibid, 2.3.12] for examples of b-divisors. 

Let $(X,S+B)$ be a plt pair such that $\lfloor S+B\rfloor=S$ is irreducible and $f:X\to Z$ a birational contraction. Then  $f$ is a {\it pl-flipping contraction} if $f$ is small, $\rho (X/Z)=1$, $-(K_X+S+B)$ is $f$-ample and $-S$ is $f$-ample. The {\it pl-flip} of $f$, if it exists, is a small birational morphism $f^+:X^+\to Z$ such that $\rho (X^+/Z)=1$, $(X^+,S^++B^+)$ is plt (where $S^++B^+$ denotes the strict transform of $S+B$), and $K_{X^+}+S^++B^+$ is $f^+$-ample. It is well known that $X^+={\rm Proj}R(X/Z,K_X+S+B)$ and thus the pl-flip $f^+$ exists if and only if $R(X/Z,K_X+S+B)=\bigoplus _{i\geq 0} f_* \mathcal O _X(i(K_X+S+B))$ is a finitely generated $\mathcal O_Z$ algebra.

We refer the reader to \cite{BCHM10} for the definitions of {\it minimal model} (which in \cite{BCHM10} is called a log terminal model), $(K_X+\Delta)$-{\it non-positive} and $(K_X+\Delta)$-{\it negative} maps.
Also see \cite{BCHM10} for Nakayama's definition of $N_{\sigma}(D)$ for a pseudo-effective divisor $D$. 

Let $f:X\to Z$ be proper morphism, and $\mathcal{F}$ be a coherent sheaf, then $\mathcal F$ is {\it relatively globally generated} if $f^*f_*\mathcal{F}\to \mathcal{F}$ is surjective. If $\mathcal{F}\cong \mathcal{O}_X (M)$ for some divisor $M$ on $Y$, then $M$ is {\it a relatively free divisor} if $\mathcal{F}$ is relatively globally generated.  
Let $C$ be a $\Q$-divisor, such that $\lceil C \rceil\geq 0$. Then a divisor $D$ is said {\it relatively $C$-saturated } if the natural injection
$$f_*\mathcal{O}_X(D)\to f_*\mathcal{O}_X(\lceil D+C \rceil)$$
is an isomorphism.
Let $L$ be a divisor on $X$, the {\it relative mobile b-divisor} ${\Mob_{Z}}(L)={\mathbf N}$, is the unique b-divisor  ${\mathbf N}$ such that for any birational morphism $g:Y\to X$, we have ${\mathbf N}_Y\le g^*L$, ${\mathbf N}_Y$ is mobile over $Z$ and the natural morphism
$(f\circ g)_*\mathbf N_Y\to f_*L $ is an isomorphism.

For a variety $X$ defined over a field $k$ of characteristic $p>0$, we always denote by $F: X\to X$ the absolute Frobenius.

\subsection{Resolution of singularities}
\begin{theorem}\label{t-res} Let $X$ be a projective variety of dimension $3$ over an algebraically closed field $k$ of characteristic $p>0$. Then there exists a nonsingular projective variety $Y$ and a birational morphism $f:Y\to X$ which is an isomorphism above the nonsingular locus $X\setminus \Sigma $.
We may assume that $f^{-1}(\Sigma)$ is a divisor with simple normal crossings. Moreover, if $\mathcal I \subset \mathcal O _X$ is an ideal and we replace $\Sigma$ by its union with the support of $\mathcal O _X/\mathcal I$, then we may assume that $\mathcal I \cdot \mathcal O _Y$ is locally principal and that $f$ is given by a sequence of blow ups along smooth centers lying over $\Sigma$. \end{theorem}
\begin{proof} See \cite{Cutkosky04}, \cite{CP08} and \cite {CP09}.\end{proof}

\subsection{$F$-singularities}\label{F-sing}

Test ideals and the theory of tight closure were introduced by Hochester-Huneke \cite{HH90}. Since then it has become increasingly clear that there is a deep connection between the classes of singularities defined in terms of Frobenius splitting properties and the ones appearing in the minimal model program. This has led to exciting progress in both commutative algebra and birational geometry (see for example the survey paper \cite{ST11} and the references therein). The definitions and results in this subsection are well-known to the experts, we include them for the reader's convenience. 

Let $(X,\Delta )$ be a pair whose index is not divisible by $p$. Choose $e>0$ so that $(p^e-1)(K_X+\Delta )$ is Cartier and let $\mathcal L _{e,\Delta}=\mathcal O _X((1-p^e)(K_X+\Delta ))$. 
Then there is a canonically determined (up to multiplication by a unit) map $$\phi _\Delta :F_*^e \mathcal L _{e,\Delta}\to \mathcal O _X.$$ The {\it non-$F$-pure ideal} $\sigma (X,\Delta )$ is the unique biggest ideal (respectively  {\it the test ideal} $\tau (X,\Delta )$ is the unique smallest nonzero ideal)  $J\subset \mathcal O _X$ such that $$(\phi _\Delta\circ F_*^e )(J\cdot \mathcal L _{e,\Delta})=J$$ for any $e>0$. If $\sigma(X,\Delta)=\mathcal{O}_X$, we say that $(X,\Delta)$ is {\it sharply $F$-pure}.  If $\tau(X,\Delta)=\mathcal{O}_X$, we say that $(X,\Delta)$ is {\it strongly $F$-regular}.  
The relationship between strongly $F$-regular and sharply $F$-pure is similar to the difference between klt and log canonical singularities. For a pair $(X,\Delta)$, we can define $F$-pure centers as in \cite{Schwede08}.
See \cite[9.5]{Schwede09} for more background.  In this note, we will only need to discuss non-trivial $F$-pure centers in the case of a simple normal crossing pair. We have the following.

\begin{lemma}\label{l-sigma}
Assume that $(X,\Delta =\sum \delta _i \Delta _i)$ is a simple normal crossing pair whose index is not divisible by $p$ and that $0\leq \delta _i\leq 1$. Then 
$\sigma (X,\Delta )=\mathcal O _X$ and each strata of
$\lfloor \Delta \rfloor$ is an $F$-pure center.\end{lemma}
\begin{proof} The first statement is \cite[15.1]{FST11} (also see \cite[6.18]{ST10}). For the rest, see the main result of \cite{HSZ10} and \cite[3.8, 3.9]{Schwede08} (see also \cite[3.5]{ST11}).
\end{proof} 
Assume $X$ is a proper variety. For any line bundle $M$ 
we define 
$$S^0(X,\sigma (X,\Delta )\otimes M)=$$
$$\bigcap _{n>0}{\rm Im}\left( H^0(X,F^{ne}_*(\sigma (X,\Delta )\otimes \mathcal L _{ne,\Delta}\otimes M^{p^{ne}}))\to H^0(X,\sigma (X,\Delta) \otimes M)\right).$$
Since $H^0(X, M)$ is a finitely dimensional vector space, we have $$S^0(X,\sigma (X,\Delta )\otimes M)={\rm Im}\left( H^0(X,F^{ne}_*( \mathcal L _{ne,\Delta}\otimes M^{p^{ne}}))\to H^0(X, M)\right)$$
for all $n\gg 0$.
Recall the following result (cf. \cite[5.1, 5.3]{Schwede11}).
\begin{prop}\label{KS} Fix $X$ a normal projective variety and suppose that $(X, \Delta )$ is a pair such that $K_X + \Delta$ is $\mathbb Q$-Cartier with index not divisible by $p > 0$. Suppose that $S \subset X$ is any union of $F$-pure centers of $(X,\Delta )$ and that $M$ is a Cartier divisor such that $M - K_X - \Delta$ is ample. Then there is a natural surjective map:
$$S^0(X,\sigma (X,\Delta)\otimes \mathcal O _X(M))\to S^0(S,\sigma (S,\phi _S^{\Delta})\otimes \mathcal O _S(M)).$$
\end{prop}

We need the following result.
\begin{lemma}\label{l-cu2}
Let $(X,\Delta )$ be a sharply $F$-pure pair and $L$ an ample Cartier divisor. Assume the index of $K_X+\Delta$ is not divisible by $p$. Then there is an integer $m_0>0$ such that for any nef Cartier divisor $P$ and any integer $m\geq m_0$, we have $$S^0(X,\sigma (X,\Delta)\otimes \mathcal O _X (mL+P))=H^0(X, \mathcal O _X (mL+P)).$$
\end{lemma}
\begin{proof}See \cite[2.23]{Patakfalvi12}.
\end{proof}
\begin{cor}\label{c-cu2}
Let $(X,\Delta )$ be an snc pair and $L$ an ample Cartier divisor. Assume the index of $K_X+\Delta$ is not divisible by $p$. Then there is an integer $m_0>0$ such that for any nef Cartier divisor $P$ and any integer $m\geq m_0$, we have $$S^0(X,\sigma (X,\Delta)\otimes \mathcal O _X (mL+P))= H^0(X,  \mathcal O _X (mL+P-\lfloor (1-\epsilon)\Delta \rfloor ))$$
for $0<\epsilon\ll 1$.
\end{cor}
\begin{proof} Let $M=\lfloor (1-\epsilon)\Delta \rfloor$. It is easy to see, using the projection formula, that 
$$S^0(X,\sigma (X,\Delta)\otimes \mathcal O _X (mL+P))\supset S^0(X,\sigma (X, \Delta-M)\otimes \mathcal O _X (mL+P-M )).$$
Since the coefficients of $\Delta-M$ are contained in $[0,1]$, we have $\sigma (X, \Delta-M)=\mathcal O _X$, i.e., $(X,\Delta-M)$ is a sharply $F$-pure pair. It then follows from \eqref{l-cu2} above that $$S^0(X,\sigma (X,\Delta-M)\otimes \mathcal O _X (mL+P-M))= H^0(X,  \mathcal O _X (mL+P-M )).$$

On the other hand, we have $(p^e-1)\Delta\ge p^eM$ for $e\gg  0$, thus by the projection formula again, we easily see that for any Cartier divisor $N$ we have $$S^0(X,\sigma (X,\Delta)\otimes \mathcal O _X (N))\subset S^0(X,\sigma(X,0)\otimes \mathcal O _X (N-M )),$$ and the reverse inclusion immediately follows. 
\end{proof}

Next we introduce a global version of strongly $F$-regular singularities.

\begin{defn}[{cf. \cite[3.1, 3.8]{SS10}}]

 Let $(X,\Delta)$ be a pair with a proper morphism $f:X\to T$ between normal varieties over an algebraically closed  field of characteristic $p>0$. Assume $X$ is normal and $\Delta$ is an effective $\Q$-divisor on $X$. The pair $(X,\Delta)$ is {\it globally $F$-regular over $T$} if for every effective divisor $D$, there exists some $e > 0$ 
such that the natural map 
$$\mathcal O_X \to F_*^e\mathcal O _X (\lceil (p^e - 1)\Delta\rceil  + D)$$ splits locally over $T$. 

When $T=X$, this definition coincides with the original definition of $(X,\Delta)$ being {\it strongly $F$-regular} and when $T={\rm Spec}(k)$, it coincides with the definition of a {\it globally $F$-regular} pair (cf. \cite{SS10}).

\end{defn}
The next result shows that the global $F$-regularity is a very restrictive condition.
\begin{thm}[Schwede-Smith]\label{t-SS10}If  $f:X\to T$ is a proper morphism of normal varieties  over an algebraically closed  field of characteristic $p>0$ and $(X,\Delta)$ is globally $F$-regular over $T$, then there is a $\Q$-divisor $\Delta '\geq \Delta $ such that the pair  $(X,\Delta ')$ is globally $F$-regular over $T$, $-(K_X+\Delta ')$ is ample over $T$ and the index of $K_X+\Delta '$ is not divisible by $p$.  
\end{thm}
\begin{proof} See \cite[4.3]{SS10} and its proof.
\end{proof}
We will need the following.
\begin{lemma}\label{l-GFR} Let $(X,\Delta)$ be a globally $F$-regular (over $T$) pair and $D$ an effective divisor. Then there exists a rational number $\epsilon >0$ such that $(X,\Delta+\epsilon D)$ is globally $F$-regular.
\end{lemma}
\begin{proof} Since $(X,\Delta)$ is globally $F$-regular, then  the map 
$$\mathcal O_X \to F_*^e\mathcal O _X (\lceil (p^e - 1)\Delta\rceil  + D)$$ splits for some $e>0$ and so $(X,\Delta+2\epsilon D)$ is globally sharply $F$-split where $2\epsilon =\frac 1 {p^e-1}$ (cf. \cite[3.1]{SS10}).
The claim is now immediate from \cite[3.9]{SS10} (applied with $C=\epsilon D)$.
\end{proof}
As in the definition of sharp $F$-purity, in this note, we will mostly work with the dual version of the above definition. 

\begin{lemma} Let $X$ be a normal variety. Let $E$ be an integral divisor. 
There is an isomorphism
{\small 
$$H^0(X, \mathcal{O}_X((1-p^e)K_X-E ))\cong {\rm Hom}_{\mathcal{O}_X}( F^e_*\mathcal{O}_X(E),\mathcal{O}_X).  $$
}
\end{lemma}
\begin{proof}Let $L=\mathcal{O}_X(E)$. We have the following equalities of sheaves
{\small
$$\mathcal{ H}om_{\mathcal{O}_X} (F^e_*L , \mathcal{O}_X )= F^e_*\mathcal{H}om_{\mathcal{O}_X} (L , (1-p^e)K_X) = F^e_*\mathcal{O}_X ((1-p^e )K_X-E).$$
}
In fact, as $X$ is normal and the above sheaves satisfy Serre's condition $S_2$, it suffices to check this along the smooth locus $X_{\rm sm}$, where it follows easily from  Grothendieck duality and the projection formula that
$$
\mathcal{ H}om_{\mathcal{O}_{X_{\rm sm}}} (F^e_*L , \mathcal{O}_X ) = \mathcal{ H}om_{\mathcal{O}_{X_{\rm sm}}} (F^e_*(L\otimes \omega _{X_{\rm sm}}^{p^e}), \omega_{X_{\rm sm}} )$$
 $$=  F^e_*\mathcal{ H}om_{\mathcal{O}_{X_{\rm sm}}} (L\otimes \omega _{X_{\rm sm}}^{p^e}, \omega_{X_{\rm sm}} )=F^e_*\mathcal{ H}om_{\mathcal{O}_{X_{\rm sm}}} (L, \omega_{X_{\rm sm}}^{\otimes (1-p^e)} ). 
$$
Taking global sections, we obtain the claim. 
\end{proof}
Applying the above lemma to $E=\lceil (p^e-1)\Delta\rceil +D$, it immediately follows that:
\begin{prop}\label{p-onto}Let $T={\rm Spec}(A)$ for some finitely generated $k$-algebra $A$ and $(X,\Delta)$ be a pair such that $X$ is proper over $T$. Then $(X,\Delta)$ is globally $F$-regular over $T$  if and only if for any effective divisor $D$, there is an integer $e\in \mathbb{N}$ and a surjection 
$$H^0(X,\mathcal{O}_X(\lfloor(1-p^e)(K_X+\Delta)\rfloor -D))\to H^0(X,\mathcal{O}_X)$$
induced by the above morphism. 
\end{prop}

\begin{prop}\label{p-bir} Let $(X,\Delta)$ be a globally $F$-regular pair (over $T$) and $f:X'\to X$ a proper birational morphism between normal varieties such that $f^*(K_X+\Delta )=K_{X'}+\Delta '$ where $\Delta '\geq 0$.
Then $(X',\Delta ')$ is globally $F$-regular.
\end{prop}
\begin{proof}  Assume that the index of $K_X+\Delta$ is not divisible by $p>0$. Let $D'$ be an effective divisor on $X'$ and pick $D$ a Cartier divisor on $X$ such that $f^*D\geq D'$. Since $(X,\Delta)$ is globally $F$-regular, we have a surjection 
$$H^0(X,\mathcal{O}_X((1-p^e)(K_X+\Delta) -D))\to H^0(X,\mathcal{O}_X)$$ for $e>0$ sufficiently divisible. 
By the projection formula, this is equivalent to the surjection 
$$H^0(X',\mathcal{O}_{X'}((1-p^e)(K_{X'}+\Delta ') -f^*D))\to H^0(X',\mathcal{O}_{X'})$$
which factors through  $H^0(X',\mathcal{O}_{X'}((1-p^e)(K_{X'}+\Delta ') -D'))$. Thus $(X',\Delta ')$ is globally $F$-regular.

To see the general case, note that we may work locally over $T$ and hence we may assume that there is a $\Q$-divisor $\Delta _1\geq \Delta$ such that $(X,\Delta _1)$ is  a globally $F$-regular pair (over $T$).
It then follows that $f^*(K_X+\Delta _1)=K_{X'}+\Delta '_1$ where $\Delta '_1\geq 0$ and  $(X',\Delta ' _1)$ is  globally $F$-regular. But then $(X',\Delta ' )$ is  globally $F$-regular since $\Delta '\leq \Delta '_1$.
\end{proof}
We will need the following easy consequence.
\begin{lemma}\label{l-gtos}
Let $f:X\to T$ be a proper birational morphism between normal varieties such that $(X,\Delta)$ is globally $F$-regular over $T$. Then $(T,\Delta_T=f_*\Delta)$ is strongly $F$-regular. 
\end{lemma}
\begin{proof}We may assume that $T$ is affine. Let $D$ be an effective divisor on $T$.
Pick $D'\geq D$ such that $D'$ is  Cartier.
 Since $(X,\Delta)$ is globally $F$-regular over $T$, we have a surjection  $$H^0(X,F_*^e\mathcal{O}_X((1-p^e)(K_X+\Delta) -f^*D'))\to H^0(X,\mathcal{O}_X).$$
Since  $H^0(X,F_*^e\mathcal{O}_X((1-p^e)(K_X+\Delta) -f^*D'))= H^0(T,f_*F_*^e\mathcal{O}_X((1-p^e)(K_X+\Delta) -f^*D'))$ is contained in $H^0(T,F_*^e\mathcal O_T((1-p^e)(K_T+\Delta_T) -D'))$, we have surjections
$$ F_*^e\mathcal{O}_T((1-p^e)(K_T+\Delta_T) -D')\to \mathcal{O}_T.$$
Since the above map factors through $F_*^e\mathcal{O}_T((1-p^e)(K_T+\Delta_T) -D)$, $(T,\Delta_T=f_*\Delta)$ is strongly $F$-regular. 
\end{proof}
We will also need the following well known perturbation lemma (cf. \cite[3.15]{Patakfalvi12}).
\begin{lem}\label{l-perturbation}
 Let $(X,\Delta)$ be a log pair, $E \geq 0$ a divisor such that $E-K_X$ is Cartier and $H={\rm Supp}(E+\Delta )$. Then for any $\epsilon >0$, we can find an effective $\Q$-Cartier divisor $D\le \epsilon H$ such that 
the $\mathbb{Q}$-Cartier index of $K_X+\Delta+D$ is not divisible by $p$.
\end{lem}
\begin{proof} Let $A$ be a Cartier divisor which is linearly equivalent to $E-K_X$. Assume that $mp^{e_0}(K_X+\Delta)$ is Cartier, where $p\nmid m\in \mathbb{N}$. Pick $e> e_0$ such that $\frac{1}{p^{e}-1}(\Delta+E)\le \epsilon H$ and let $D=\frac{1}{p^{e}-1}(\Delta+E)$. Then 
\begin{eqnarray*}
m(p^{e}-1)(K_X+\Delta+D) & = & m(p^{e}-1)(K_X+\frac{p^{e}}{p^{e}-1}  \Delta+\frac{1}{p^{e}-1}E)\\
 &=&mp^{e}\Delta+m(p^{e}-1)K_X+mE\\
 &\sim&mp^{e}\Delta+m(p^{e}K_X+A),
\end{eqnarray*}
is Cartier. 
\end{proof}
\subsection{Stabilization of $S^0$ in the relative case}
The results of this section were communicated to us by Karl Schwede (cf. \cite{Schwede13}). We thank him for allowing us to include them in this note.

Suppose that $(X, \Delta \geq 0  )$ is a pair such that $L_{g,\Delta }=(1-p^g )(K_X + \Delta )$ is Cartier for some $g>0$,  $f : X \to  Y$ is a projective morphism with $Y$ normal and $M$ is a Cartier divisor on $X$.
\begin{defn}With the above notation, $S^0f_*(\sigma (X,\Delta ) \otimes \mathcal O_X(M))$ is defined to be the intersection:
$$\bigcap _{e\geq 0}{\rm Image}\left( {\rm Tr}^{eg}F^{eg}_*f_*\mathcal O _X((1-p^{eg})(K_X +\Delta )+p^{eg}M)\to f_*\mathcal O_X (M)\right) $$
 which is more compactly written as $$\bigcap _{e\geq 0}{\rm Image}\left( {\rm Tr}^{eg}F^{eg}_*f_*\mathcal O _X(L_{eg,\Delta }+p^{eg}M)\to f_*\mathcal O_X (M)\right) .$$
\end{defn}
This intersection is a descending intersection, so a priori it needs not stabilize. Therefore, it is unclear if it is a coherent sheaf. At least when $M - K_X - \Delta$ is ample, we show now that this is the case.
\begin{prop}\label{p-stab}
If $M - K_X - \Delta$ is $f$-ample, then
$$S^0f_*(\sigma (X,\Delta ) \otimes \mathcal O_X(M))={\rm Image}\left( {\rm Tr}^{eg}F^{eg}_*f_*\mathcal O _X(L_{eg,\Delta }+p^{eg}M)\to f_*\mathcal O_X (M)\right) $$ for $e\gg 0$. In particular, it is a coherent sheaf.
\end{prop}
\begin{proof} It is easy to see that $S^0f_*(\sigma (X,\Delta ) \otimes  \mathcal O_X(M))$ is equal to $$\bigcap _{e\geq 0}{\rm Image}\left( {\rm Tr}^{eg}F^{eg}_*f_*(\sigma (X,\Delta ) \otimes  \mathcal O _X(L_{eg,\Delta }+p^{eg}M))\to f_*(\sigma (X,\Delta ) \otimes  \mathcal O_X (M))\right) .$$
Let $\mathcal K$ denote the kernel of the surjective map $$F^{g}_*(\sigma (X,\Delta ) \otimes \mathcal O _X(L_{g,\Delta }))\to \sigma (X,\Delta ) .$$
By Serre Vanishing there is an integer $e_0>0$ such that $$R^1f_*(\mathcal K\otimes \mathcal O _X (p^{eg}M +L_{eg,\Delta }) )=0$$ for all $e\geq e_0$.
It follows that the map $$ F_*^{(1+e)g}f_*(\sigma (X,\Delta )\otimes \mathcal O_X(p^{(1+e)g}M+L_{(1+e)g,\Delta })) \to  F_*^{eg}f_*(\sigma (X,\Delta )\otimes \mathcal O_X(p^{eg}M+L_{eg,\Delta }))$$ is surjective for all $e\geq e_0$ and hence so are the maps $$ F_*^{(d+e)g}f_*(\sigma (X,\Delta )\otimes \mathcal O_X(p^{(d+e)g}M+L_{(d+e)g,\Delta })) \to  F_*^{eg}f_*(\sigma (X,\Delta )\otimes \mathcal O_X(p^{eg}M+L_{eg,\Delta }))$$ for all $e\geq e_0$
and $d\geq 1$. Therefore, these maps have the same images under the trace. In other words,
$${\rm Image}\left( {\rm Tr}^{eg}F^{eg}_*f_*\mathcal O _X(L_{eg,\Delta }+p^{eg}M)\to f_*\mathcal O_X (M)\right) =$$
 $${\rm Image}\left( {\rm Tr}^{e_0g}F^{e_0g}_*f_*\mathcal O _X(L_{e_0g,\Delta }+p^{e_0g}M)\to f_*\mathcal O_X (M)\right) $$ for all $e\geq e_0$ and the proposition is proven.
\end{proof}
\begin{rmk}\label{r-fS} If $Y$ is affine, then we can identify $F^{eg}_*f_*\mathcal O _X((1-p^{eg})(K_X +\Delta )+p^{eg}M)$ with $F^{eg}_*H^0(X,\mathcal O _X((1-p^{eg})(K_X +\Delta )+p^{eg}M))$
and hence $S^0f_*(\sigma (X,\Delta )\otimes \mathcal O _X(M))$ with $S^0(X,\sigma (X,\Delta )\otimes \mathcal O _X(M))$. By \eqref{p-onto}, it follows that if $(X,\Delta )$ is $F$-regular over $Y$, then $S^0f_*(\sigma (X,\Delta ))=f_* \mathcal O_X$ (the stabilization in this case is automatic and does not require that $(p^g-1)(K_X+\Delta )$ is Cartier for some $g>0$; however see \eqref{t-SS10}).
\end{rmk}
\begin{rmk}\label{r-rel} It is easy to see that if $M - K_X - \Delta$ is $f$-ample, then the results of Section 2.3 easily translate to corresponding results about the sheaves $S^0f_*(\sigma (X,\Delta ) \otimes \mathcal O_X(M))$. In what follows we will explicitely state only the results that will be frequently used in what follows.
\end{rmk}
\begin{prop}\label{p-rKS} With the above notation, assume that $p$ does not divide the index of $K_X+\Delta$, $M$ is a Cartier divisor, $M-(K_X+\Delta )$ is $f$-ample and $S\subset X$ is any union of $F$-pure centers of $(X,\Delta )$. Then there is a natural surjective map 
$$S^0f_*(\sigma (X,\Delta ) \otimes \mathcal O_X(M))\to S^0(f|_S)_*(\sigma (S,\phi ^\Delta _S ) \otimes \mathcal O_S(M)).$$
\end{prop}
\begin{proof} The proof is immediate from the arguments in the proof of \cite[5.3]{Schwede11}.\end{proof}
\begin{lemma}\label{l-rcu2}  Let $L$ be an $f$-ample Cartier divisor and $(X,\Delta )$ a sharply $F$-pure pair such that $p$ does not divide the index of $K_X+\Delta$. Then there exists an integer $m_0>0$ such that for any $f$-nef Cartier divisor $P$ and any integer $m\geq m_0$, we have $$S^0f_*(\sigma (X,\Delta)\otimes \mathcal O _X (mL+P))=f_* \mathcal O _X (mL+P).$$
\end{lemma}
\begin{proof} The proof is immediate from the arguments in the proof of  \cite[2.23]{Patakfalvi12}.\end{proof}
\begin{lemma}\label{l-rcu3} Let $f:X\to Y$ be a proper morphism of normal varieties such that $(X,\Delta )$ is $F$-regular over $Y$. If $A$ is a sufficiently ample divisor on $Y$, then $$S^0(X,\sigma (X,\Delta)\otimes \mathcal O _X (f^* A))=H^0(X, \mathcal O _X(f^*A)).$$
\end{lemma}
\begin{proof}  Since $(X,\Delta )$ is $F$-regular over $Y$, by \eqref{t-SS10}, we may pick $\Delta '\geq \Delta$ such that   $(X,\Delta ')$ is also $F$-regular over $Y$, the index of $K_X+\Delta '$ is not divisible by $p$ and $-(K_X+\Delta ')$ is ample over $Y$. 
Since $(X,\Delta ')$ is $F$-regular over $Y$, $(X,\Delta ')$ is strongly $F$-regular so that $\sigma (X,\Delta ')=\mathcal O_X$. Fix $g>0$ such that $(p^g-1)(K_X+\Delta ')$ is Cartier, and $$F^{eg}_*\mathcal L_{eg,\Delta '}\to \mathcal O_X,\ {\rm and}\qquad f_*F^{eg}_*\mathcal L_{eg,\Delta '}\to f_* \mathcal O_X$$ are surjective for any $e>0$. Let $\mathcal K$ be the kernel of the map $$F^{g}_*\mathcal L_{g,\Delta '}\to \mathcal O_X.$$ 
Since $(F^{g}_*\mathcal L_{g,\Delta '})\otimes \mathcal L_{(e-1)g,\Delta '}=F^{g}_*\mathcal L_{eg,\Delta '}$, twisting by $\mathcal L_{(e-1)g,\Delta '}$ and pushing forward by $F_*^{(e-1)g}$ we obtain the short exact sequence
$$0\to F_*^{(e-1)g}(\mathcal K\otimes \mathcal L_{(e-1)g,\Delta '})\to F_*^{eg}\mathcal L_{eg,\Delta '}\to   F_*^{(e-1)g} \mathcal L_{(e-1)g,\Delta '}\to0.$$
Since $-(K_X+\Delta ')$ is $f$-ample, there exists $e_0>0$ such that $R^1f_*(\mathcal K\otimes \mathcal L_{(e-1)g,\Delta '})=0$ for all $e\geq e_0$.
Since $(X,\Delta ')$ is $F$-regular over $Y$, pushing forward via $f$ we obtain the short exact sequences
$$0\to F_*^{(e-1)g}f_*(\mathcal K\otimes \mathcal L_{(e-1)g,\Delta '})\to F_*^{eg}f_* \mathcal L_{eg,\Delta '}\to   F_*^{(e-1)g}f_* \mathcal L_{(e-1)g,\Delta '}\to0.$$
If $e\gg 0$ and $A$ is sufficiently ample, then $ L_{(e-1)g,\Delta '}+f^*p^{(e-1)g}A$ is sufficiently ample so that by Fujita vanishing (see \cite{Keeler03}), we have that $$H^1(Y, F_*^{(e-1)g} f_*(\mathcal K\otimes \mathcal L_{(e-1)g,\Delta '})\otimes \mathcal O _Y(A))=$$
$$H^1(Y, f_*(\mathcal K\otimes \mathcal L_{(e-1)g,\Delta '}\otimes \mathcal O _X(p^{(e-1)g}f^*A)))=$$
$$H^1(X,\mathcal K\otimes \mathcal L_{(e-1)g,\Delta '}\otimes f^*\mathcal O _Y(p^{(e-1)g}A)))=0.$$ We may also assume that $$H^0(Y,f_* F^{e_0g}_*\mathcal L_{e_0g,\Delta '}\otimes \mathcal O_Y(A))\to H^0(Y,f_*\mathcal O _X\otimes \mathcal O_Y(A))$$ is surjective. But then since $$H^0(Y,f_* F_*^{eg}(\mathcal L_{eg,\Delta '}\otimes \mathcal O _X(p^{eg}f^*A)))\to  H^0(Y,f_* F_*^{(e-1)g}(\mathcal L_{(e-1)g,\Delta '}\otimes \mathcal O _X(p^{(e-1)g}f^*A)))$$ is surjective for all $e>e_0$, it follows that $$H^0(Y,f_* F_*^{eg}(\mathcal L_{eg,\Delta '}\otimes \mathcal O _X(p^{eg}f^*A)))\to  
H^0(Y,f_*\mathcal O _X\otimes \mathcal O_Y(A))$$ is surjective for any $e\geq e_0$ and hence $$S^0f(X,\sigma (X,\Delta ')\otimes \mathcal O _X (f^* A))=H^0(Y, \mathcal O _Y(A)).$$ Since $ S^0(X,\sigma (X,\Delta ')\otimes \mathcal O _X (f^*A))\subset S^0(X,\sigma (X,\Delta )\otimes \mathcal O _X (f^*A))$ the lemma is proven.
\end{proof}
\subsection{Surfaces}
In this subsection, we collect the results in MMP theory for surfaces (in characteristic $p>0$) that we will need later.
\begin{prop}\label{l-sab}
Let $f:S\to R$ be a projective morphism from a normal surface, and $B_S$ be a $\Q$-divisor on $S$ such that $(S,B_S)$ is a relative weak log Fano surface, i.e., $(S,B_S)$ is klt and $-(K_S+B_S)$ is $f$-ample. Then we have the following 
\begin{enumerate}
\item any relatively nef divisor is semi-ample over $R$,
\item the nef cone of $S$ (over $R$) is finitely generated.
\end{enumerate}
\end{prop}
\begin{proof}See \cite[15.2]{Tanaka12a} and \cite[3.2]{Tanaka12b}.
\end{proof}

\begin{lemma}\label{l-0}Let $(S,B_S)$ be a klt pair. There exists a unique birational morphism $\nu :\bar S\to S$ such that 
\begin{enumerate}
\item $K_{\bar S}+B_{\bar S}= \nu ^*(K_S+B_S)$, where $B_{\bar S}\geq 0$, and
\item $(\bar{S},B_{\bar{S}})$ is terminal.
\end{enumerate}
$({\bar S},B_{\bar S})$ is the terminalization of $( S,B_S)$. In particular $\bar S$ is smooth and ${\rm mult }_p(B_{\bar S})<1$ for all $p\in \bar S$.\end{lemma}
\begin{proof} We reproduce the following well known proof. For any log resolution of $g_{S'}: S'\to S$, we write $$g_{S'}^*(K_S+B_S)+E=K_{S'}+B_{S'} ,$$
where $E$, $B_{S'}$ are effective and have no common components. Passing to a higher resolution, we can assume $(S',B_{S'})$ is terminal. 

We then run the minimal model program for $(S',B_{S'})$ over $S$ and we obtain a relative minimal model $\mu:S'\to \bar{S}$ with a morphism $\nu: \bar{S}\to S$ (see \cite{KK}, \cite{Tanaka12a}). Note that $\mu_*E\sim_{S,\mathbb{Q}}K_{\bar{S}}+B_{\bar{S}}$ is nef where $B_{\bar{S}}=\mu_*B_{S'}$. Since $\mu_*E$ is an exceptional curve, its self-intersection is non-positive and if $\mu_*E\ne 0$ the self intersection is negative. Thus $\mu_*E=0$ and hence $K_{\bar{S}}+B_{\bar{S}}\sim_{\Q} \nu^*(K_S+B_S)$. As $(S',B_{S'})$ is terminal, so is $(\bar{S},B_{\bar{S}})$. Uniqueness is also well known, and we omit the proof. 
\end{proof}

Even though the Kawamata-Viehweg vanishing theorem does not hold for surfaces, it is still true for a birational morphism between surfaces.

\begin{lemma}[{\cite[2.2.5]{KK}}]\label{l-KV}
 Let $h : S' \to S$ be a proper birational morphism between normal surfaces, such that $(S',B_{S'})$ is a klt pair.  Let $L$ be a Cartier divisor on $S'$, and $N$ an $h$-nef $\Q$-divisor such that $L\equiv K_{S'}+B_{S'}+N$. Then $R^1h_*\mathcal{O}_{S'}(L)=0$.
\end{lemma}
\begin{proof}{\cite[2.25]{KK}} proves the case when $S'$ is smooth. In general, we can take the terminalization of $(S,B_S)$ and use a simple spectral sequence argument. 
\end{proof}

\section{On the $F$-regularity of weak log del Pezzo surfaces}\label{S-surface}
Hara proved that a klt surface $S$ is strongly $F$-regular if the characteristic is larger than 5 (see \cite{Hara98}). The aim of this section is to generalize Hara's result to establish the global $F$-regularity for relative weak log del Pezzo surfaces of birational type with standard coefficients when the characteristic is larger than 5.  We will use Shokurov's theory of {\it complements} (see \cite{Shokurov00, Prokhorov99}), which fits in this context very well. 

\begin{thm}\label{t-gfr} Assume the ground field $k$ is algebraically closed of characteristic  $p>5$. Let $(S,B)$ be a pair with a birational  proper morphism $f: S\to T$ on to a normal surface germ $(T,0)$ such that \begin{enumerate}\item $(S,B)$ is klt,
\item  $ -(K_S+B) $ is $f$-nef, and
\item the coefficients of $B$ are in the standard set $\{\frac{n-1}{n}| n\in \mathbb{N}\}$.\end{enumerate}Then $(S,B)$ is globally $F$-regular over $T$.\end{thm}
We will need the following result on complements due to Shokurov.
\begin{thm}\label{t-c} There exists a divisor $B^c\ge B$ and an integer $N\in \mathcal{R}N_2=\{1,2,3,4,6\}$, such that $N(K_S+B^c)\sim 0$ over $T$, and $(S,B^c)$ is log canonical but not klt. If $(S,B^c)$ is not plt, then we may assume that $N\in \{1,2\}$. \end{thm}
\begin{proof} The proof follows immediately from \cite[6.0.6]{Prokhorov99}. The fact that  $B^c\ge B$,  follows easily since $NB^c\geq \lfloor (N+1)B\rfloor$ and the  coefficients of $B$ are in the standard set (cf. \cite[4.1.3]{Prokhorov99}). The proof in \cite{Prokhorov99} only uses the Kawamata-Viehweg vanishing theorem for morphisms between birational models over $T$, which is also known to hold in characteristic $p>0$ (see \eqref{l-KV}).
\end{proof} 
\begin{proof}[Proof of \eqref{t-gfr}.]We will assume that  $N$ is minimal as above. 
Let $\nu :\tilde S\to S$ be a smooth dlt model of $(S,B^c)$ (so that $(\tilde S, B_{\tilde S}^c)$ has dlt singularities) and write 
$$K_{\tilde S}+B^c_{\tilde S}=\nu ^*(K_S+B^c)\qquad \mbox{and}  \qquad K_{\tilde S}+B_{\tilde S}=\nu ^*(K_S+B).$$ 
In particular  $B^c_{\tilde S}\geq 0 $. Note however that $ B_{\tilde S}$ is possibly not  effective.
We assume first that we are in the {\it exceptional case} that is $(S, B^c)$ is plt. We let $ C=\lfloor B^c_{\tilde S} \rfloor $. 

Consider the extended dual graph $G$ (see \cite[2.26]{Kollar13} for the definition) corresponding to the exceptional curves for $\mu =f\circ \nu$  and the strict transform of $B^c_T=f_*B^c$. Let $G_C$ be the subgraph constructed by first removing all vertices such that the corresponding curve appears in $B_{\tilde S}^c-B_{\tilde S}$  with coefficient $0$ and then discarding all connected components not containing the vertex corresponding to $C$.  Note that by the arguments in \cite{Prokhorov99}, the curve $C$ is exceptional over $T$ and as $\tilde S$ dominates the minimal resolution and $T$ is log terminal, all exceptional curves are smooth rational curves meeting transversely. Note also that by adjunction we have 
that $2=\sum _{i\in I}\frac {p_i}{N}$ where $\frac {p_i}{N}$ is the coefficient of a non-zero component of $B^c_{\tilde S} $ and the sum is over all of those components that intersect $C$. Since $0<\frac {p_i}{N}<1$, we have $|I|\geq 3$. If $|I|>3$, then $|I|=4$ and each coefficient equals $\frac 1 2$ and so $N=2$ (cf. \S 4.1 of  \cite{Prokhorov99}). Otherwise we have that $|I|=3$, i.e. that $C$ is a fork with exactly three branches in which case the $G_C$ looks like this
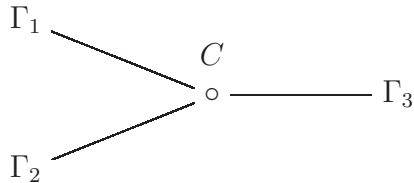
\begin{diagram}[height=1.2em, w=1.8em]
\Gamma_1             \\
                    &   \rdLine  & C\\
                    &                &\circ  &\rLine &\Gamma_3   \\
                   & \ldLine \\
\Gamma_2        
\end{diagram}
where each $\Gamma_i$ is a connected tree.

Let $\Psi_i$ (resp. $\Psi_i^c$) be the subdivisor of ${B}_{\tilde S}$ (resp. ${B}^c_{\tilde S}$) consisting of components in ${B}_{\tilde S}$ contained in the support of $\Gamma _i$. Then $${B}_{\tilde S}=cC+\sum^m_{i=1}\Psi_i+\Lambda _{\tilde S} , \ {\rm and}\qquad {B}_{\tilde S}^c=C+\sum^m_{i=1}\Psi_i^c+\Lambda ^c_{\tilde S}$$ where $c<1$, $\Lambda _{\tilde S}$ and $\Lambda ^c_{\tilde S}$ are the remaining components (whose support is contained in the components corresponding to $G-G_C$) and $m=|I|\in\{3, 4\}$. Note that $\Psi_i^c\geq 0$ but the $\Psi_i$ are not necessarily effective.
We say that a chain is {\it non-exceptional} if it has at least one component which is not exceptional.

\begin{lemma}\label{l-small}
There exists an $i=i_0$, and a non-exceptional chain $\Gamma _i^0\subset \Gamma _i$ intersecting $C$ such that ${\rm Supp} (\Psi_i^c-\Psi_i)\supset \Gamma ^0_i$.
\end{lemma}
\begin{proof}
Let $D$ be the connected component of $(1-c)C+\sum (\Psi _{i}^c-\Psi _{i})$  containing $C$. Note that the support of $D$ contains $C$ and so $D$ is non-zero.  If $D$ is exceptional, then it is easy to see that 
$$(B_{\tilde S}^c-B_{\tilde S})\cdot D=D^2<0.$$  This is impossible as $$B_{\tilde S}^c-B_{\tilde S}=K_{\tilde S}+B_{\tilde S}^c-(K_{\tilde S}+B_{\tilde S})=\nu ^*(K_S+B_{S}^c-K_S-B_{S})\sim_{\Q,T}-\nu ^*(K_S+B_S)$$ is nef over $T$.

\end{proof}
We assume that $i_0=1$, we let $\Gamma ^0_1$ be the subchain of $\Gamma _1$  defined in \eqref{l-small}. We assume that the first vertex is a neighbor of $C$ and the last vertex corresponds to a non-exceptional curve in the support of $B^c_{\tilde S}-B_{\tilde S}$.
We define the divisor 
$${B}^*_{\tilde S}=C+\Psi _1^*+\sum^m_{i=2}\Psi ^c_i+\Lambda^c_{\tilde{S}},$$ where we replace each coefficient $\frac{p}{q}$ of a component of $\Psi _1^c\wedge \Gamma _1^0$ with $q=N\in \{1,2,3,4,6\}$ by the coefficient $\frac{p-1}{q-1}$ and we leave the remaining coefficients unchanged. 
\begin{lemma}\label{l-nef} $-(K_{\tilde S}+B^*_{\tilde S})$ is nef over $T$.\end{lemma}
\begin{proof} Since $B^*_{\tilde S}\leq B^c_{\tilde S}$, it is clear that $(K_{\tilde S}+B^*_{\tilde S})\cdot C<0$.
If $D$ is an exceptional  curve in $\Gamma _1\setminus \Gamma _1^0$, then by the same argument, it is clear that $(K_{\tilde S}+B^*_{\tilde S})\cdot D\leq 0$.

For any exceptional curve $D$ contained in $\Gamma ^0_1$ which does not meet the non-exceptional component, we know that its intersection number with the adjacent components in $\Gamma_1^0$ are all  1. Thus the  equation $(K_{\tilde S}+B^c_{\tilde S})\cdot D=0$ implies that
$$\frac {p_{j-1}}q+\frac {p_{j+1}}q+\frac {p_{j}}q D^2+\frac r q-2-D^2=0\qquad {\rm or}$$
$$p_{j-1}+p_{j+1}+r-2q+(p_j-q)D^2=0.$$
(We have used the fact that by adjunction one has $K_{\tilde S}\cdot D=-2-D^2$.)
Here we denote by $p_j/q$ the multiplicity of $D$ in $B^c_{\tilde S}$ and by $p_{j-1}/q$ (resp. $p_{j+1}/q$) the multiplicity of $B^c_{\tilde S}$ along the previous (resp. the following) curve in $\Gamma ^i_0$. If $D$ is the first curve in $\Gamma ^0_1$, then we let $j-1=0$ and $p_{0}=q$ as the multiplicity of $B^c_{\tilde S}$ along $C$ is $1$.
$r/q$ denotes the intersection of all other components of $B^c_{\tilde S}$ with $D$.
To obtain $B^*_{\tilde S}$, we replace $(p_{j-1},p_j,p_{j+1},q)$ by $(p_{j-1}-1, p_j-1,p_{j+1}-1,q-1)$, note that $p_0=q$ and hence $(p_0-1)/(q-1)=1$ so that the multiplicity along $C$ is unchanged.
The above equality implies 
$$\frac {p_{j-1}-1}{q-1}+\frac {p_{j+1}-1}{q-1}+\frac {p_{j}-1}{q-1} D^2+\frac r q-2-D^2\le 0$$
which says $(K_{\tilde S}+B^*_{\tilde S})\cdot D\le 0$.

For the exceptional curve $D\in \Gamma_1^0$ which meets the non-exceptional component say $G$, the same calculation implies that $(K_{\tilde S}+B^*_{\tilde S})\cdot D=0$ if the intersection number is $D\cdot G=1$, and otherwise $(K_{\tilde S}+B^*_{\tilde S})\cdot D<0$. 
\end{proof} 
\begin{lemma}\label{l-geq} We have $B^*_{\tilde S}\geq B_{\tilde S}$.\end{lemma}
\begin{proof} Since $B_{\tilde S}-B^*_{\tilde S}=K_{\tilde S}+B_{\tilde S}-(K_{\tilde S}+B^*_{\tilde S})$ is nef over $S$, by the negativity lemma it suffices to show that $\nu _*B^*_{\tilde S}\geq B$.
This follows by \eqref{l-trivial} below, since $B^c- B\ge0$, the support of $B^c- B$ contains $\Gamma^0_i$ and the coefficients of $B$ are of the form $\frac {n-1} n$. 
\end{proof}
\begin{lemma}\label{l-trivial}
Let $p,q, j$ be natural numbers such that  $\frac{j-1}{j}<\frac{p}{q}$ and $q\ge 2$, then $\frac{j-1}{j}\leq \frac{p-1}{q-1}$.
\end{lemma}
\begin{proof} Obvious.\end{proof}

\begin{lemma} We have that $(C,{\rm Diff }_C(B^*_{\tilde S}))$ is globally $F$-regular.
\end{lemma}
\begin{proof} We know that $N\in \{ 1,2,3,4,6 \}$. If $N=6$, then by simple arithmetic we have $(C,{\rm Diff }_C(B^c_{\tilde S}))=(\mathbb{P}^1, \frac{1}{2}P_1+\frac{2}{3}P_2+\frac{5}{6}P_3)$. But then $(C,{\rm Diff }_C(B^*_{\tilde S}))=(\mathbb{P}^1, aP_1+bP_2+cP_3)$ where $$(a,b,c)\in \{ (2/5,2/3,5/6), (1/2, 1/2,5/6), (1/2, 2/3, 4/5)\}.$$
If $N=4$, then $(C,{\rm Diff }_C(B^c_{\tilde S}))=(\mathbb{P}^1, \frac{1}{2}P_1+\frac{3}{4}P_2+\frac{3}{4}P_3)$. But then $(C,{\rm Diff }_C(B^*_{\tilde S}))=(\mathbb{P}^1, aP_1+bP_2+cP_3)$ where $$(a,b,c)\in \{ (1/3,3/4,3/4), (1/2, 2/3,3/4)\}.$$
If $N=3$, then $(C,{\rm Diff }_C(B^c_{\tilde S}))=(\mathbb{P}^1, \frac{2}{3}P_1+\frac{2}{3}P_2+\frac{2}{3}P_3)$. But then $(C,{\rm Diff }_C(B^*_{\tilde S}))=(\mathbb{P}^1, \frac 1 2P_1+\frac 2 3 P_2+\frac 2 3P_3)$.

$N=2$, then $(C,{\rm Diff }_C(B^c_{\tilde S}))=(\mathbb{P}^1, \frac{1}{2}P_1+\frac{1}{2}P_2+\frac{1}{2}P_3+\frac{1}{2}P_4)$. But then $(C,{\rm Diff }_C(B^*_{\tilde S}))=(\mathbb{P}^1, \frac 1 2P_1+\frac 1 2 P_2+\frac 1 2P_3)$

All of these cases are globally $F$-regular by \eqref{p-hara}. 
\end{proof}
\begin{prop}\label{p-Freg}Notation as above. 
If there is an effective $\mathbb Q$-divisor ${B}^*_{\tilde S}$ on ${\tilde S}$ such that
\begin{enumerate}
\item $-(K_{\tilde {S}}+B^*_{\tilde S})$ is nef, $B^c_{\tilde S}\ge B^*_{\tilde S}\ge {B}_{\tilde S}$,
\item $(\tilde S, B^*_{\tilde{S}})$ is plt, 
\item $(C,{\rm Diff}_C(B^*_{\tilde S}))$ is globally $F$-regular. 
\end{enumerate}
Then $(S,B)$ is globally $F$-regular over $T$.
\end{prop}
\begin{proof} The argument is a generalization of \cite[4.3]{Hara98}. 

It follows from our assumption that $-(K_{\tilde{S}}+B^*_{\tilde{S}})$ is semi-ample  over $T$. We denote by
$$\hat{S}:={\rm Proj}R(\tilde{S}/T,-(K_{\tilde{S}}+B^*_{\tilde{S}})).$$
 Let $F$ be an exceptional relative anti-ample  divisor for $\tilde{S}$ over $\hat{S}$. We note that by (3), $-(K_{\tilde {S}}+B^*_{\tilde S})\cdot C>0$ and hence $\tilde{S}\to \hat{S}$ does not contract $C$.
 
 We choose $0<\epsilon \ll 1$ a rational number such that $(\tilde{S},B^*_{\tilde{S}}+\epsilon F)$ is plt and  $(C,{\rm Diff}_C(B^*_{\tilde S}+\epsilon F))$ is globally $F$-regular.  Let $E$ be any effective integral divisor over $T$.  We write $E=n_0C+E'$ where the support of $E'$ does not contain $C$.  We have the following commutative diagram.
  \begin{diagram}
H^0(\tilde S, \mathcal{O}_{\tilde S}((1-p^e)(K_{\tilde{S}}+B^*_{\tilde{S}}+\epsilon F)-E' )) & \rTo_{\alpha}& H^0(\tilde S, \mathcal{O}_{\tilde S})\\
\dTo ^{\gamma} & &\dTo _{\beta}\\
H^0(C , \mathcal{O}_C((1-p^e) (K_C+{\rm Diff}_C(B^*_{\tilde S}+\epsilon F))-E'|_{C})) &\rTo ^{\xi} &  H^0(C,\mathcal{O}_C)
\end{diagram}
By assumption $(C,{\rm Diff}_C(B^*_{\tilde S}+\epsilon F))$ is globally $F$-regular, hence $\xi$ is a surjection for $e\gg 0$. Since, for $0<\epsilon \ll 1$, $-(K_{\tilde{S}}+B^*_{\tilde{S}}+\epsilon F)$ is ample over $T$, it follows that for any $e\gg 0$, the homomorphism $\gamma$ is also a surjection by the Kawamata-Viehweg vanishing theorem \eqref{l-KV}. Thus $\beta \circ \alpha$ is a surjection and hence so is $\alpha$ by Nakayama's lemma (since $\beta$ is given by $ A\to A/m$ where $A\cong H^0(\tilde S, \mathcal{O}_{\tilde S})$ and $m$ is the maximal ideal of the image of $C$ in $T$).

Finally, since  ${\rm mult}_C(B^*_{\tilde{S}})> {\rm mult}_C(B_{\tilde{S}})$, then for any $e\gg 0$, the image of $\alpha$ is contained in the image of 
$$H^0(\tilde S, \mathcal{O}_{\tilde S}((1-p^e)(K_{\tilde{S}}+B_{\tilde{S}})+E)) \to H^0(\tilde S, \mathcal{O}_{\tilde S}).$$
\end{proof}
We now consider the case that $(S,B^c)$ is not plt and hence $N\in \{ 1, 2\}$. The proof is similar but easier to the plt case and so we only indicate the necessary changes to the above arguments.
Notice that since $(T,B_T:=f_*B)$ is klt, then $B^c_T:=f_*B^c\ne B_T$. 
Let $\nu :\tilde S \to S$ be a smooth dlt model, then 
by the Koll\'ar-Shokurov connectedness theorem, which follows from the Kawamata-Viehweg vanishing theorem in this case as usual, we know that the  components in $B^c_{\tilde{S}}$ with coefficient 1 form a connected graph. Computing the intersection numbers, we know that this graph is a chain $\Gamma^{\bullet}$, with the property that only the two end points of $\Gamma^{\bullet}$ can be connected to components in ${\rm Supp}(B^c_{\tilde{S}})\setminus \Gamma^{\bullet}$.

Let $C$ be any component in $\Gamma^{\bullet}$. By the same argument as \eqref{l-small}, there is a non-exceptional chain $\Gamma'' \subset (\Gamma \setminus C)$ intersecting $C$ such that ${\rm Supp}(B^c_{\tilde S}-B_{\tilde S})\supset \Gamma ''$. Thus, $\Gamma^{\bullet}\cup \Gamma ''$ is a chain, with a non-exceptional component at one end. 
If the other end is exceptional, which implies it is in $\Gamma^{\bullet}$, then we may and will rechoose it to be $C$.  
 We define $\Gamma' =(\Gamma''\cup \Gamma^\bullet)\setminus \{C\}$, which is one chain if $C$ is an end and two chains otherwise. In particular, $\Gamma'\subset {\rm Supp}(B^c_{\tilde S}-B_{\tilde S})$ and $C$ meets every connected component of $\Gamma'$ at one of its ends.  

Write $B^{c}_{\tilde{S}}=B'+B''$, where $B'$ consists of the components contained in the support of $\Gamma'$. For all exceptional curves $C_i\ne C$ in $\Gamma '$, consider the equations 
$$(K_{\tilde{S}}+B''+\sum a_iC_i)\cdot C_i=0,$$ 

Since the intersection matrix $C_i\cdot C_j$ is negative definite, we have a unique solution.
We easily see that 
$${\rm Supp }(B'-\sum a_iC_i)={\rm Supp} (\Gamma').$$

Then we first define a curve $B^{\sharp}_{\tilde{S}}$ as follows:  for a component not in $\Gamma'$, its coefficient in $B^{\sharp}_{\tilde{S}}$ is the same as in $B^c_{\tilde{S}}$, for a component in $\Gamma'$, its coefficient is the same as in $\sum a_iC_i$. So we easily see
$$(K_{\tilde{S}}+B^{\sharp}_{\tilde{S}})\cdot D\leq 0,$$
for any exceptional curve $D$.

Now we define $B^*_{\tilde{S}}=(1-\epsilon)B^c_{\tilde{S}}+\epsilon B^{\sharp}_{\tilde{S}}$. Since $B^{\sharp}_{\tilde{S}}\ge 0$, if we choose $\epsilon$ sufficiently small, we see that $B^*_{\tilde{S}}\ge B_{\tilde{S}}\vee 0$  and $(\tilde{S},B^*_{\tilde{S}})$ is plt. 

Note that if $(K_{\tilde{S}}+B_{\tilde S}^c)|_C=K_{\mathbb P ^1}+\sum _{i=1}^jb_iP_i$, then as each $b_i\geq \frac 12$, we have $j\leq 4$ and if $j=4$, then $b_i=\frac 12$. This is impossible as we assume $(\tilde{S}, B^c_{\tilde{S}})$ is non-plt.
Therefore we may assume that $j=2$ or $3$. 
Thus
{\small 
$$(K_{\tilde{S}}+B^*_{\tilde{S}})|_C=K_{\mathbb{P}^1}+a_1P_1+a_2P_1+a_3P_3 \mbox{ or } (K_{\tilde{S}}+B^*_{\tilde{S}})|_C=K_{\mathbb{P}^1}+a_1P_1+a_2P_1,$$
}where in the first case we have $(a_1,a_2)\le (\frac{1}{2},\frac{1}{2})$ and $a_3<1$ and in the second case $a_1,a_2<1$.
By \eqref{p-Freg}, $(S,B)$ is globally $F$-regular over $T$.
\end{proof}

\begin{prop}\label{p-hara} Let $k$ be an algebraically closed field of characteristic $p>5$ and $D= \frac 2 5 P_1+\frac 2 3 P_2+\frac 5 6 P_3$ or $D= \frac 1 3 P_1+\frac  34 P_2+\frac 3 4 P_3$ or $D=\sum _{i=1}^r\frac {d_i-1}{d_i}P_i$, where $r\leq 2$ or $r=3$ and $$(d_1,d_2,d_3)\in \{ (2,2,d), (2,3,3), (2,3,4), (2,3,5)\}$$ then $(\mathbb P ^1,D)$ is globally $F$-regular. 
\end{prop} 
\begin{proof} By \cite[4.2]{Watanabe91}, we may assume that $D_1= \frac 2 5 P_1+\frac 2 3 P_2+\frac 5 6 P_3$ or $D_2= \frac 1 3 P_1+\frac  34 P_2+\frac 3 4 P_3$.

We may assume that $P_1=0$, $P_2=\infty$ and $P_3=1$.
By Fedder's criterion for pairs, it is enough to check that if $D=c_1P_1+c_2P_2+c_3P_3$ for some $e>0$ if $a_i=\lceil (p^e-1)c_i\rceil$, then $x^{a_1}y^{a_2}(x+y)^{a_3}$ contains a monomial $x^iy^j$ where $i, j < p^e - 1$. (Note in fact that the pair $(\mathbb{P}^1,D')$ is strongly $F$-regular where $D'=\sum \frac {a_i}{p^e-1}P_i$ and $D'\geq D$.)

\noindent{\bf Case 1: $D_1= \frac 2 5 P_1+\frac 2 3 P_2+\frac 5 6 P_3$}.

Let $e=1$. Since
$$(p-1)D'_1=\lceil (p-1)D_1\rceil\leq (\frac{2}{5}(p-1)+\frac 45) P_1+(\frac 2 3(p-1)+\frac 2 3) P_2+(\frac 5 6(p-1)+\frac 5 6) P_3 ,$$
then $a_1+a_2+a_3<2p-3$ for any $p >34$. One also sees that the same inequality works for $p=31$.
In these cases the monomial $x^iy^j$ has non-zero coefficient for $j=\lfloor \frac {a_1+a_2+a_3}2\rfloor <p-1$ and
$i=a_1+a_2+a_3-j<p-1$.
Therefore, we only need to check the cases $p\in \{ 7,11, 13,17,19,23, 29 \}$.

When $p=7$ and $e=2$, we have 
$a_1=20$, $a_2=32$, $a_3=40,$
and $x^{20} y^{32}(x+y)^{40}$ has the nonzero term $x^{46}y^{46}$.

When $p=11$ and $e=2$, we have 
$a_1=48$, $a_2=80$, and $a_3=100 ,$
thus $x^{48} y^{80}(x+y)^{100}$ has the nonzero term $x^{115}y^{113}$.

When $p=13$ and $e=2$, we have 
$a_1=68$, $a_2=112$, $a_3=140,$
and $x^{68} y^{112}(x+y)^{140}$ has the nonzero term $x^{166}y^{154}$.

When $p=17$ and $e=2$, we have 
$a_1=116$, $a_2=192$ and $a_3=240 $,
thus $x^{116} y^{192}(x+y)^{240}$ has the nonzero term $x^{287}y^{261}$.

When $p=19$ and $e=2$, we have 
$a_1=144$, $a_2=240$ and $a_3=300 $,
thus $x^{144} y^{240}(x+y)^{300}$ has the nonzero term $x^{357}y^{327}$.

When $p=23$ and $e=2$, we have 
$a_1=212$, $a_2=352$ and $a_3=440 $,
thus $x^{212} y^{352}(x+y)^{440}$ has the nonzero term $x^{491}y^{513}$.

When $p=29$ and  $e=2$, we have 
$a_1=336$, $a_2= 560$ and $a_3 =700$,
thus $x^{336} y^{560}(x+y)^{700}$ has the nonzero term $x^{775}y^{821}$. 

\noindent{\bf Case 2: $D_2= \frac 1 3 P_1+\frac  34 P_2+\frac 3 4 P_3$}.

Similarly, if $e=1$,
$$(p-1)D_2'=\lceil (p-1)D_2\rceil\leq (\frac{1}{3}(p-1)+\frac 2 3) P_1+(\frac 3 4(p-1)+\frac 3 4) P_2+(\frac 3 4(p-1)+\frac 3 4) P_3 ,$$ then $a_1+a_2+a_3<2p-3$ for any 
 $p>20$. One also sees that the same inequality works for $p\in\{ 13,17, 19\}$. In these cases the monomial $x^iy^j$ has non-zero coefficient for $j=\lfloor \frac {a_1+a_2+a_3}2\rfloor$ and
$i=a_1+a_2+a_3-j$.
Therefore, we only need to check for $p\in \{7,11\}$.

When $p=7$ and $e=2$, we have 
$a_1=16$, $a_2= 36$ and $ a_3=36$.
Thus $x^{16} y^{36}(x+y)^{36}$ has the nonzero term $x^{44}y^{44}$.

When $p=11$ and $e=2$, we have 
$a_1=40$, $a_2= 90$ and $a_3=90$.
Thus $x^{40} y^{90}(x+y)^{90}$ has nonzero term $x^{108}y^{112}$.
\end{proof}

\section{Existence of pl-flips}

\subsection{Normality of plt centers}
 In characteristic 0, by a result of Kawamata, we know that plt centers (or more generally minimal log canonical centers) are normal. The proof uses the Kawamata-Viehweg vanishing theorem.  The analogous result in characteristic $p>0$ is not known. We prove  a related result. The argument illustrates some of the techniques that will be used in the rest of this section. 

\begin{prop}\label{p-normal}
Let $(X,S+B)$ be a  plt pair with $S=\lfloor S+B \rfloor$, $S^n\to S$ the normalization and write $(K_X+S+B)|_{S^n}=K_{S^n}+B_{S^n}$. If $(S^n,B_{S^n})$ is strongly $F$-regular, and $(X,S+B)$ has a log resolution $g:Y\to X$ which supports an exceptional $g$-ample $\mathbb{Q}$-divisor $-F$. Then $S$ is normal. 
\end{prop}
\begin{proof} 
We define  
$$K_Y+S'=g^*(K_X+S+B)+\mathbf A_Y \qquad {\rm and}\qquad \mathbf A_{S'}=\mathbf{A}_Y|_{S'},$$ 
where $S'$ is the birational transform of $S$.

  We  choose an effective $\mathbb{Q}$-Cartier $\mathbb{Q}$-divisor $A$ on $X$ whose support contains the divisor $E$ defined in \eqref{l-perturbation}, $g({\rm Ex}(g))$ and ${\rm Supp}(B)$ but not ${\rm Supp}(S)$, such that $(S^n,{B}^* _{S^n}=B_{S^n}+ A|_{S^n})$ is strongly $F$-regular (cf. \eqref{l-GFR}). We then pick $\Xi$ such that 
\begin{enumerate}
\item $S'+\{-\mathbf{A}_{Y}\}\le \Xi\le S'+\{-\mathbf{A}_{Y}\}+ g^*A$,
\item the index of $K_Y+\Xi$ is not divisible by $p$, and
\item $\lceil \mathbf{A}_Y\rceil-(K_{Y}+\Xi )\sim _\Q -g^*(K_{X}+S+B)-(\Xi-S'-\{-\mathbf{A}_{Y}\})$ is $g$-ample.
\end{enumerate}
To construct such a $\mathbb{Q}$-divisor $\Xi$, we proceed as follows: Let $\Xi'=S'+\{-\mathbf{A}_{Y}\}+\epsilon F$ for some $0<\epsilon \ll 1 $ so that (1) holds. Since $\lceil \mathbf{A}_Y\rceil-(K_{Y}+\Xi ' )\sim _\Q -g^*(K_{X}+S+B)-\epsilon F$ is $g$-ample for $0<\epsilon \ll 1$, we may assume that (3) also holds. It is easy to see that the support of $g^*A$ contains an effective divisor $E'$ such that $E'-K_Y$ is Cartier and so we may use \eqref{l-perturbation} to slightly increase the coefficients of $\Xi '$ to obtain $\Xi$ so that (1-3) are satisfied.  

By \eqref{p-rKS}, there is a surjection 
$$S^0g_*(\sigma (Y, \Xi )\otimes\mathcal O _Y( \lceil \mathbf{A}_Y\rceil ))\to S^0h_*(\sigma(S',\Xi _{S'}) \otimes \mathcal O _{S'}(\lceil \mathbf{A}_{S'}\rceil ))$$
where $\Xi_{S'}=(\Xi -S') |_{S'}$ and $h:S'\to S^n$ is the induced morphism. We claim that $$S^0h_*(\sigma(S',\Xi _{S'}) \otimes \mathcal O _{S'}(\lceil \mathbf{A}_{S'}\rceil))= \mathcal O _{S^n}.$$ Grant this for the time being,
then since $$S^0g_*(\sigma (Y, \Xi )\otimes\mathcal O _Y( \lceil \mathbf{A}_Y\rceil ))\subset g_*\mathcal O_Y(\lceil \mathbf{A}_Y\rceil)=\mathcal O _X$$ 
(as $\lceil \mathbf{A}_Y\rceil$ is effective and exceptional),
it follows that
the homomorphism $\mathcal O_X\to  \mathcal O_{S^n}$ is surjective. This implies that $S=S^n$. 

To see the claim, since $\{ -\mathbf{A}_Y\}-\lceil \mathbf{A}_Y\rceil=-\mathbf{A}_Y$, note that the second inequality in (1) implies (after adding $-\lceil \mathbf{A}_Y\rceil$ and restricting to $S'$)  that
$$h^*(K_{S^n}+ B_{S^n}+A|_{S^n})\ge K_{S'}+\Xi _{S'}-\lceil \mathbf A_{S'}\rceil .$$
Since ${B}^* _{S^n}=B_{S^n}+ A|_{S^n}$, it then follows that 
$$(1-p^e)h^*(K_{S^n} +B _{S^n}^* )\leq (1-p^e)(K_{S'}+\Xi _{S'})+p^e\lceil \mathbf{A}_{S'}\rceil $$ and so
{
 $$F^e_*\mathcal O_{S^n}((1-p^e)(K_{S^n}+B_{S^n}^* ) )\subset h_*F^e_*\mathcal{O}_{S'}((1-p^e)(K_{S'}+\Xi _{S'})+p^e\lceil \mathbf{A}_{S'}\rceil).$$
}
Consider the following  commutative diagram.
\begin{diagram}
F^e_*\mathcal{O}_{S^n}((1-p^e)(K_{S^n} +B _{S^n} ^* ))&\rTo & \mathcal{O}_{S^n}\\
\dTo & & \dTo {\cong}\\
h_*F^e_*\mathcal{O}_{S'}((1-p^e)(K_{S'}+\Xi _{S'})+p^e\lceil \mathbf{A}_{S'}\rceil)&\rTo & h_*\mathcal{O}_{S'}(\lceil \mathbf{A}_{S'}\rceil)
\end{diagram}
Since $(S^n,{B}^* _{S^n}=B_{S^n}+ A|_{S^n})$ is strongly $F$-regular, we have that
$ \sigma(S^n,B_{S^n}^* )= \mathcal O_{S^n}$
(cf. \eqref{l-rcu2}) and hence the top horizontal arrow is surjective.
It then follows that the bottom horizontal arrow is surjective so that the claim follows. 
 \end{proof}

\subsection{b-divisors} Assume that $f:(X,S+B)\to Z$ is a 3-fold pl-flipping contraction so that
\begin{enumerate}\item $X$ is a normal threefold,
\item $f:X\to Z$ is a small projective birational contraction with $\rho (X/Z)=1$,
\item $(X,S+B)$ is plt, 
and 
\item $-S$ and $-(K_X+S+B)$ are ample over $Z$. 
\end{enumerate} 
We will assume that $Z$ is affine. In particular any divisor which is ample over $Z$ is in fact ample and if $\mathcal F$ is a coherent sheaf on $X$, then we may identify $f_*\mathcal F$ and $H^0(X,\mathcal F)$.

Let $A\geq 0$ be an auxiliary $\mathbb{Q}$-divisor on $X$ whose support contains the divisor $E$ defined in \eqref{l-perturbation}, ${\rm Supp}(B)$ and ${\rm Ex}(f)$ but not ${\rm Supp} (S)$, such that $(X,S+B+A)$ is plt. Let $S^n\to S$ be the normalization of $S$. Let $B_{S^n}={\rm Diff}_{S^n}(X,B)$ and $B^*_{S^n}={\rm Diff}_{S^n}(X,B+ A)$, so that
$$(K_X+S+B)|_{S^n}=K_{S^n}+B_{S^n}\qquad \mbox{and}\qquad (K_X+S+B+ A)|_{S^n}=K_{S^n}+B^{*}_{S^n}.$$

Let $g:Y\to X$ be a log resolution of $(X,S+A)$ (cf. \eqref{t-res}) and $h=f\circ g:Y\to Z$ the composition.
 For any divisor $D$ on $X$, let $D'=g^{-1}_*D$  be the strict transform on $Y$. 
 We assume that the restriction $g_{S'}:S'\to S^n$ factors through  the terminal model $\bar{S}$ of $(S^n,B_{S^n})$ given by \eqref{l-sab}, i.e.,  $g_{S'}=\nu \circ \mu$ for a morphism $\mu:S'\to \bar S$.
 We  write $$K_Y+S'=g^*(K_X+S+B)+\mathbf A_Y,\qquad K_{S'}=g_{S'}^*(K_{S^n}+B_{S^n})+\mathbf A_{S'}$$
where  $\mathbf A_{S'}=(\mathbf A _Y)|_{S'}$. Here, by abuse of notation, we use $\mathbf{A}$ to mean the discrepancy b-divisors of $(X,S+B)$ and of $(S^n, B_{S^n})$ (see \cite[2.3.12(3)]{Corti07}).  As the restriction of the discrepancy b-divisor of $(X,S+B)$ to the models of $S$ gives the discrepancy b-divisor of  $(S^n, B_{S^n})$, this should not cause any confusion.
After possibly blowing up further, we assume there exists $F\geq 0$ a $g$-exceptional $\mathbb{Q}$-divisor on $Y$ such that $-F$ is relatively ample.
As in the proof of \eqref{p-normal}, we can pick $\Xi$ such that 
\begin{enumerate}
\item $S'+\{-\mathbf{A}_{Y}\}\le \Xi\le S'+\{-\mathbf{A}_{Y}\}+ g^*A$,
\item the index of $K_Y+\Xi$ is not divisible by $p$, and
\item $\lceil \mathbf A _Y\rceil -(K_{Y}+\Xi )$ is ample.
\end{enumerate}

\begin{lemma}\label{l-2} 
 Let $M$ be a relatively free $\mathbf A_Y$-saturated divisor on $Y/Z$ (cf. Subsection \eqref{NC}). 
Then there exists $M_Y\in |M|$ such that
$M_{S'}:=M_Y|_{S'}$ 
descends to $\bar S$ i.e. $M_{S'}=\mu ^* M_{\bar S}$ where $M_{\bar S}=\mu _* M_{S'}$. 
\end{lemma}
\begin{proof} Since $Z$ is affine, $|M|$ is base point free. Since $f\circ g$ and $(f\circ g)|_{S'}$ are birational, $|M|$ induces a birational morphism $\psi:Y\to \mathbb{P}^N$, whose restriction to $S'$ is also birational. Note that there is a big open subset $U\subset \psi (S')$ (the complement of finitely many points) such that $\psi|_{S'}$ is finite over $U$. 
Let $D$ be a general hyperplane divisor on $\psi (Y)$ and $M_Y\in |M| $ the corresponding divisor.
We note that as we are in characteristic $p>0$, $M_Y$ is not necessarily smooth.
Since $D\cap \psi (S')$ is contained in $U$, we have that  $M_{S'}=M_Y|_{ S'}\to \psi (M_{S'})$ is finite and
for any $x\in  M_{S'}$, there exists $M'_{S'}\sim M_{S'}$ (obtained by considering another general hyperplane $\psi (x)\in D'\subset \psi(Y)$) such that 
$x\in  M_{S'}$ and $M_{S'}\wedge M'_{S'}=0$.

By \eqref{l-sigma}, ${S'}$ is an $F$-pure center for 
$(Y,\Xi)$.
If we let $L=M_Y+\lceil \mathbf A_Y \rceil $, then since $M_Y$ is nef,
$L-(K_Y+\Xi)$ is ample (cf. (3) above). Thus,  
by \eqref{KS},  we have that
$$S^0(Y,  \sigma (Y,\Xi) \otimes \mathcal O_Y(L))\to  \noindent S^0({S'}, \sigma ({S'},\Xi _{{S'}})\otimes \mathcal O _{S'}(L|_{{S'}}))
$$
is surjective, where $\Xi_{S'}=\phi _{S'}^\Xi=(\Xi -S')|_{S'}$. 

$\Xi _{{S'}}$ has simple normal crossings on $S'$. 
Since $M$ is relatively $\mathbf{A}_Y$-saturated, for $m\gg 0$ and any $s \in H^0(Y, \mathcal O_Y(L))$, we have that $s$ vanishes along $\lceil \mathbf A_{Y} \rceil$. It follows that any section $t\in S^0({S'}, \sigma ({S'},\Xi_{S'})\otimes \mathcal O _{S'}(L|_{{S'}}))$ vanishes along $(\lceil \mathbf A_{Y} \rceil )|_{S'}=\lceil \mathbf A_{S'}\rceil $.
\begin{claim} We may assume that $M_{S'}$ smooth on a neighborhood of the support of $\lceil \mathbf A_{S'}\rceil$.\end{claim}
\begin{proof} Suppose that $M_{S'}$ is singular at some point $x$ of the support of $\lceil \mathbf A_{S'}\rceil$. 
Let $m={\rm mult }_x M_{S'}\geq 2$. By  choosing $M_{S'}$ sufficiently generally, we may assume that there exists  $M_{S'}'\sim M_{S'}$ with $m={\rm mult }_x M_{S'}'$ and $M_{S'}\wedge M'_{S'}=0$. Let $\pi:S''\to S'$ be the blow up of $S'$ at $x$ with an exceptional divisor $E$. It follows easily that $\pi ^*M_{S'}-mE$ is nef and hence the Seshadri constant of $M_{S'}$ at $x$ satisfies $\epsilon (x,M_{S'})\geq 2$.
Since $L|_{S'}-(K_{S'}+\Xi_{S'} )-M_{S'}$ is ample and $(S',\Xi_{S'})$ is $F$-regular, by \cite{MS12}, we have that the Frobenius-Seshadri constant satisfies $$\epsilon _F(x, L|_{S'}-(K_{S'}+\Xi_{S'} ))>\epsilon _F(x,M_{S'})\geq \frac 12 \epsilon (x,M_{S'})\geq 1.$$
Following the proof of \cite[3.1]{MS12}, we have that the image of
$$H^0(S',F^e_*\mathcal O _{S'}(p^eL|_{S'}+(1-p^e)(K_{S'}+\Xi _{S'})))\to H^0(S',\mathcal O_{S'}(L|_{S'}))$$ contains a section $\sigma$ not vanishing at $x$.  But any such section is in $S^0({S'}, \sigma ({S'},\Xi_{S'})\otimes \mathcal{O}_{S'}(L|_{S'}) )$. As we have observed above, any such $\sigma$ must vanish along $\lceil \mathbf A _{S'}\rceil$ and so $x\not \in {\rm Supp}(\lceil \mathbf A _{S'}\rceil)$.
\end{proof}

After further blowing up along centers which do not intersect with the support of $\lceil \mathbf A _{S'}\rceil$, we may assume that $M_{S'}=M_1+M_2$ where $M_1$ denotes the $Z$-horizontal components of $M_{S'}$ and $M_2$ the $Z$-vertical components. We may further assume that   $M_1$ is smooth.

By \eqref{l-sigma},  $M_1$ is a disjoint union of $F$-pure centers for $(Y,\Xi+M_Y)$. We let $\Gamma=( \Xi_{S'}+M_2)|_{M_1}$. Arguing as above, we have a surjective map
{\small
$$S^0(Y,  \sigma (Y,\Xi +M_Y) \otimes \mathcal O_Y(L))\to  \noindent S^0(M_1, \sigma (M_1, \Gamma)\otimes \mathcal O_{M_1}(L|_{M_1})).$$
}
Notice that on a neighborhood of $\lceil \mathbf A _{S'}\rceil $, $\Gamma$ is equal to $( \Xi-S' )|_{M_1}$.
We have that
{\small
$$S^0(M_1,  \sigma (M_1,\Gamma ) \otimes \mathcal O_{M_1}( L|_{M_1})) \supset    S^0(M_1,  \sigma (M_1, \{ \Gamma \} ) \otimes \mathcal O_{M_1}(L|_{M_1}-\lfloor \Gamma\rfloor) )$$
} (see for example the proof of \eqref{c-cu2}).
As $M_1$ is affine, by \eqref{l-rcu2}, it then follows that $\mathcal O _{M_1}(L|_{M_1}-\lfloor \Gamma \rfloor)$ is globally generated by $S^0(M_1,  \sigma (M_1, \{ \Gamma \} ) \otimes \mathcal O_{M_1}(L|_{M_1}-\lfloor \Gamma\rfloor) )$
on a neighborhood of $\lceil \mathbf A _{S'}\rceil $. 
Since 
$$S^0(Y,  \sigma (Y,\Xi +M_Y) \otimes \mathcal O_Y(L)) \subset S^0(Y,  \sigma (Y, \Xi ) \otimes \mathcal O_Y(L)) ,$$ 
any section 
$$\sigma \in S^0(M_1, \sigma (M_1, \Gamma)\otimes \mathcal O_{M_1}(L|_{M_1}))$$
 lifts to a section in $S^0(Y,  \sigma (Y,\Xi ) \otimes \mathcal O_Y(L))$ and hence $\sigma $ vanishes along $(\lceil \mathbf A _{S'}\rceil )|_{M_1} $. However, if $P$ is a point contained in the support of $(\lceil \mathbf A _{S'}\rceil )|_{M_1} $, then
since we may assume that the coefficients of $\Xi-S'-\{-\mathbf A_Y \}$ are sufficiently small, we have 
 $${\rm mult}_P(\lfloor \Gamma\rfloor )={\rm mult}_P(\lfloor  ( \Xi-{S'} )|_{M_1}\rfloor  )< {\rm mult}_P(( \lceil \mathbf A _{S'}\rceil )|_{M_1} ),$$
 which is a contradiction. 
  This implies that $(\lceil \mathbf A _{S'}\rceil )|_{M_{S'}}=(\lceil \mathbf A _{S'}\rceil )|_{M_1} =0$. Since the support of $\lceil \mathbf A _{S'}\rceil $ is the $S'\to \bar S$ exceptional locus, it follows that $M_{S'}$ descends to $\bar S$.
\end{proof}

We now define b-divisors as in \cite[2.4.1]{Corti07}. We fix an effective Cartier divisor $Q\sim k(K_X+S+B)$ on $X$, such that 
the support of $ Q$ does not contain $S$.
Let
$$\mathbf N _i={\Mob}(iQ),\qquad \mathbf M_i=\mathbf N _i|_S,\  {\rm and}\qquad \mathbf D_i=\frac 1 i \mathbf M_i .$$ 
Note that by \eqref{t-res}, for any $i>0$ there exists $g:Y\to X$ (depending on $i$) such that $\mathbf N _{i,Y}$ is free and hence $\mathbf N _i$ descends to $Y$. 
We may assume that $Y\to X$ is a log resolution of $Q$ and $|iQ|$ so that $\mathbf N _{i,Y}$ has simple normal crossings.

\begin{lemma}\label{l-desc} With the above notation $\mathbf M_i$ descends to $\bar S$.
\end{lemma}

\begin{proof} For any integer $i>0$, we can choose a log resolution $g:Y\to X$ of the pair $(X,S+B)$ and of the  linear system $|iQ|$. Thus, we can write $g^*(iQ)=N_i+F_i$, where $N_i\sim \mathbf{N}_{i,Y}$ is a free divisor. The divisor  $N_i$ is relatively $\lceil \mathbf A_Y\rceil$-saturated, since the inclusions
$$(f\circ g)_*\mathcal O _Y(N_i)\to(f\circ g)_*\mathcal O _Y(N_i+\lceil \mathbf{A}_{Y}\rceil)\to (f\circ g)_*\mathcal O _Y(g^*(iQ)+\lceil \mathbf{A}_{Y}\rceil)$$
are isomorphisms. The result now follows from \eqref{l-2}.
\end{proof}

In what follows, we will  fix a model $g_0:Y_0\to X$ and the birational transform $S_0\subset Y_0$ such that $S_ 0$ admits a morphism $\mu_0:S_0 \to \bar{S}$. We also assume that the models $g:Y\to X$,  factor through $Y_0$ and that $\rho:Y\to Y_0$ is an isomorphism over $X\setminus {\rm Ex}(f)$ (cf. \eqref{t-res}).

{\footnotesize
\begin{diagram}
& & S' & \rTo &Y\\
& & \dTo_{\rho_{S'}} & & \dTo_{\rho}\\
& & S_0&\rTo & Y_0\\
&\ldTo_{\mu_0} & \dTo & & \dTo_{g_0}\\
\bar{S}&\rTo^{\nu}& S^n& \rTo&X&\rTo^{f}& Z
\end{diagram}
}

\begin{lemma}\label{l-prel}
For any effective $\mathbb{Q}$-Cartier $\mathbb{Q}$-divisor $G$ on $\bar S$ we may fix an effective $\mathbb{Q}$-divisor $A^*$ on $X$ and rational numbers $0<\epsilon '\ll \epsilon $, and for any $Y$ as above we may choose $F\geq 0$ on $Y$ such that
\begin{enumerate}
\item the support of $A^*$ contains the supports of $B$, $Q$, ${\rm Ex}(f)$ and the divisor $E$ defined in \eqref{l-perturbation} but not $S$, 
\item  $-g^*(K_X+S+B)- \epsilon F+\rho ^* A' $ is ample for any $-\epsilon ' g_0^*A^*\leq A'\leq \epsilon'  g_0^*A^*$, 
\item  $-(g^*(K_X+S+B)+ \epsilon F)|_{S'} +\mu ^*A'$ is ample for any $-\epsilon '\nu ^*(A^*|_{S^n})\leq A'\leq \epsilon '\nu ^*(A^*|_{S^n})$, and 
\item if the support of $G$ contains $\nu ^{-1}(f({\rm Ex}(f))\cup g_0(F_0))$, then $\epsilon F|_{S'} \le \mu ^*(G)$. 
\end{enumerate}\end{lemma}
\begin{proof} (1) is immediate. To see (2), notice that we may assume that there is an effective $g_0$ exceptional divisor $F_0$ on $Y_0$ such that $-F_0$ is ample over $X$. It then follows that $-g_0^*(K_X+S+B)-\epsilon F_0$ is ample for $0<\epsilon \ll 1$ and it is easy to see that $-g_0^*(K_X+S+B)- \epsilon F_0+A' $ is ample for any $-2\epsilon ' g_0^*A^*\leq A'\leq 2\epsilon'  g_0^*A^*$ and $0<\epsilon '\ll \epsilon $. 
Let $F_1$ be a $\rho$ exceptional divisor on $Y$ such that $-F_1$ is ample over $Y_0$.
(2) now follows (by an easy compactness argument), letting $F=\rho ^* F_0+\lambda F_1$ where $0<\lambda \ll 1$. The proof of (3) is similar.
We also easily see that if $\epsilon$ and $\lambda$ are sufficiently small, then 
$$\epsilon F|_{S'}=\epsilon\rho ^*( F_0|_{S_0})+\epsilon\lambda F_1|_{S'} \le\frac{1}{2} \mu^*G+\frac{1}{2}\mu^*G=\mu^*G.$$
Note that the choice of $\lambda $ depends on $Y$, but the choice of $\epsilon $ does not.
\end{proof}

Note that the support of $g^*A^*$ contains the $g$ exceptional locus. 
For any $i,j>0$, we define
$$L_{ij}=\lceil \frac ji\mathbf N _{i,Y}+\mathbf A _Y\rceil.$$ 

\begin{lemma}\label{l-psi}
For any $0<\epsilon '\ll \epsilon \ll 1$, we can pick a $\Q$-divisor $\Psi$ on $Y$ (depending on $i,j$ and $Y$), with the following properties:
\begin{enumerate}
\item $\Psi'_\epsilon \leq \Psi\leq \Psi'_\epsilon  +g^*A^*$, where $\Psi'_\epsilon =\{ -\frac ji\mathbf N _{i,Y}-\mathbf A _Y\}+S'+\epsilon F $,
\item the index of $K_Y+\Psi$ is not divisible by $p$,
\item $L_{ij}-(K_Y+\Psi)+\rho ^* A'$ is ample for any $-\epsilon 'g_0^*A^*\leq A'\leq \epsilon ' g_0^*A^*$, and
\item $(L_{ij}-(K_Y+\Psi))|_{S'}-\mu ^* M$ is ample for any $\Q$-divisor $M$ on $\bar S$ such that $-\epsilon ' \nu ^*(A^*|_{ S^n})\leq M-\frac j i \mathbf M _{i,\bar S}\leq \epsilon '\nu ^*(A^*|_{ S^n})$.
\end{enumerate}\end{lemma}
\begin{proof}To construct such a $\Psi$, we proceed as follows: By \eqref{l-perturbation}, for any $0<\epsilon ''\ll 1$, we may pick a $\Q$-divisor $\Psi '_\epsilon  \leq \Psi \leq \Psi_\epsilon +\epsilon ''g^*A^*$ such that 
(1) and (2) hold.
Since\begin{eqnarray}
L_{ij}-(K_Y+\Psi) &=&\lceil \frac ji\mathbf N _{i,Y}+\mathbf A _Y\rceil-(K_Y+\Psi )\nonumber
\\ 
&= &\dfrac{j}{i}\mathbf N _{i,Y}+\mathbf A _Y- (\Psi -\Psi'_\epsilon )-K_Y-S'-\epsilon F\nonumber
\\
&=&\dfrac{j}{i}\mathbf N _{i,Y}-g^*(K_X+S+B)-\epsilon F- (\Psi -\Psi'_\epsilon ),\nonumber
\end{eqnarray}
$\mathbf N _{i,Y}$ is nef and  $0<\epsilon ''\ll 1$, (3) follows from (2) of \eqref{l-prel}.

Since $\mathbf M_{i,S'}=\mu ^*\mathbf M_{i,\bar S}$, (4) easily follows from (3) of \eqref{l-prel} and the equality {\footnotesize$$(L_{ij}-(K_Y+\Psi ))|_{S'}-\mu ^* M=-(g^*(K_X+S+B)+\epsilon F+\Psi -\Psi '_\epsilon )|_{S'}+\mu ^* (\frac j i \mathbf M_{i,\bar S}-M).$$}\end{proof}

\begin{lemma}\label{YtoS} 
The homomorphism
$$S^0(Y,\sigma (Y,\Psi)\otimes \mathcal O _Y(L_{ij}))\to S^0(S',\sigma (S', \Psi _{S'})\otimes \mathcal O _{S'} (L_{ij}|_{S'}))
$$
is surjective.
\end{lemma}

\begin{proof} Since $L_{ij}-(K_Y+\Psi) $ is ample (cf. (3) of \eqref{l-psi}), the surjectivity follows from \eqref{p-rKS} and \eqref{l-sigma}.
\end{proof}

\subsection{Rationality of $\mathbf{D}$}Let $\mathbf{D}=\lim_i \mathbf{D}_i $.

\begin{lemma}The $\mathbb R$-divisor $\mathbf {D}_{\bar S}$ is semi-ample over $f(S)$.

\end{lemma}
\begin{proof}
See \eqref{l-sab} and the argument in \cite[2.4.11]{Corti07}. 
\end{proof}

Let $h:\bar{S}\to f(S)$ be the induced morphism. Let $V\subset {\rm Div}(\bar S)\otimes \mathbb{R}$ be the smallest linear subspace defined over $\mathbb{Q}$ containing $\mathbf D _{\bar S}$.
Since the nef cone of $\bar S$ over $f(S)$ is finitely generated, we may pick nef divisors $M_i\in V$ such that $\mathbf D _{\bar S}$ is contained in the convex cone generated by $M_i$ and if $\Sigma$ is any $h$-exceptional curve, then $M_i\cdot \Sigma=0$ iff $\mathbf D _{\bar S}\cdot \Sigma=0$. By \eqref{l-sab}, after replacing each $M_i$ by a positive multiple, we may assume that each $|M_i|$ defines a birational morphism to a normal surface  over $f(S)$, say 
$a_i:\bar S\to S^i$. Notice that the exceptional set of $a_i$ corresponds to 
the set of $h$-exceptional curves intersecting $M_i$ trivially, and so this set is independent of $i$. Therefore
the morphism $a_i$ is independent of $i$ and we denote it by $a:\bar S\to S^+$. 

By Diophantine approximation,  we may pick $j>0$ and $M \in V$ such that (cf. \cite[2.4.12]{Corti07})
$M=\sum a_iM_i$ where $a_i\in \mathbb N$,
and $||M-j\mathbf  D _{\bar S}||\leq \frac {\epsilon '} 2$.  
It follows that  $M$ is relatively base point free and the map defined by $|M|$ is also given by $a:\bar S\to S^+$.

Note that if $\Sigma$ is any proper curve  contained in ${\rm Ex}(h)$, then $\Sigma \cdot \mathbf{D}_{\bar{S}}=0$ if and only if $\Sigma \cdot M=0$.
Thus $\mathbf{D}$ descends to $S^+$.  We can assume that $C$ is a smooth general curve such that $C\sim  M$. To see this, note that there is a big open subset $U$ of $S^+$ (the complement of finitely many points) such that $U$ is smooth and $a$ is an isomorphism over $U$. $C$ is then isomorphic to a general hyperplane $C^+$ of $S^+$, which is contained in $U$.
We may also assume that $\Psi _{S'}+C'$ has simple normal crossings,  where $C'$ is the strict transform of $C$ on $S'$.

\begin{lemma}\label{StoC} Let  $\Theta = \Psi _{S'}+C'$. 
 If $j$ is as above and $i\gg 0 $ is divisible by $j$, then
$$S^0(S',\sigma (S', \Theta )\otimes \mathcal O _{S'} ( L_{ij}))\to S^0(C',\sigma (C',\Theta _{C'})\otimes \mathcal O _{C'} (L_{ij}))
$$
is surjective, where $\Theta_{C'}=(\Theta-C')|_{C'}$.
\end{lemma}
\begin{proof} Recall that $\Psi _{S'}+C'$ has simple normal crossings and since $-\epsilon '\nu ^* A^*|_{S^n}\leq \frac j i\mathbf M _{i,\bar S}-M \leq \epsilon '\nu ^* A^*|_{S^n}$ (for $i\gg 0$) it follows that $L_{ij}|_{S'}- (K_{S'}+\Theta)$ is ample (cf. (4) of \eqref{l-psi}).
The lemma now follows from \eqref{p-rKS}.

\end{proof}

\begin{lemma}\label{l-ex} If $j$ is as above and $i\gg 0 $ is divisible by $j$, then $\lceil \frac ji\mathbf M _{i,S' }+\mathbf A_{S'}\rceil -\mathbf M _{j,S' }$ is $a$-exceptional.
\end{lemma}
\begin{proof}Combining \eqref{YtoS} and \eqref{StoC} (and the fact that $\Theta \geq \Psi _{S'}$)  it follows that, for any $m\gg 1$, the image of 
$$H^0(Y, \mathcal O _{Y} (L_{ij} ))\to H^0(C', \mathcal O _{C'} (L_{ij}))$$
contains the subspace $ S^0(C',\sigma (C',\Theta _{C'})\otimes \mathcal O _{C'} (L_{ij} ))$ which generates $\mathcal{O}_{C'}(L_{ij} ) $ (cf. \eqref{l-rcu2}).

On the other hand, since $\lceil \frac ji\mathbf N _{i,Y}+\mathbf A _Y\rceil \leq jg^*Q+\lceil \mathbf A _Y\rceil$ and $\lceil \mathbf A _Y\rceil$ is effective and exceptional, we have an isomorphism
 $$(f \circ g)_* \mathcal O _{Y} ( \mathbf{N}_{j,Y} )\to (f \circ g)_* \mathcal O _{Y} (\lceil \frac j i\mathbf N_{i,Y}+\mathbf A _{Y}\rceil  ),$$
 which is induced by adding the effective divisor $\lceil \frac ji\mathbf N _{i,Y}+\mathbf A _Y\rceil-\mathbf N_{j,Y}$. 
 Therefore, we conclude that the sections in the image of
 $$H^0(Y, \mathcal O _{Y} (L_{ij} )) \to H^0(C', \mathcal O _{C'} (L_{ij} ))$$ vanish along $(\lceil \frac ji\mathbf N _{i,Y}+\mathbf A _Y\rceil-\mathbf N_{j,Y})|_{C'}$. 
Since, as we have seen above, they also generate $\mathcal O _{C'} (L_{ij} )$, we must have $(\lceil \frac ji\mathbf N _{i,Y}+\mathbf A _Y\rceil-\mathbf N_{j,Y})|_{C'}=0$. The lemma now follows.
\end{proof}

\begin{corollary}\label{c-rat} 
$\mathbf D_{\bar S}$ is rational and $a_*\mathbf D_{\bar S}=a_*\mathbf D_{ j,\bar S}$ for some $j>0$.\end{corollary}
\begin{proof} Suppose that $a_*\mathbf D_{\bar S}$ is not rational, then arguing as in \cite[2.4.12]{Corti07}, we may assume that the divisor $a_*(j\mathbf D_{\bar S})-a_*M$ is not effective.
 Since $\lceil \mathbf{A}_{\bar{S}} \rceil= 0$, we may pick $\delta >0$ such that the coefficients of $\mathbf{A}_{\bar{S}} $ are greater than $\delta -1$. We may assume that $\delta >\epsilon '/2$ and hence for $i\gg 0$, we have $||M- \frac ji\mathbf M _{i, \bar S}||<\delta$. By an easy computation, one sees that $M\le  \lceil \frac ji\mathbf M _{i, \bar S}+\mathbf A _{\bar S}\rceil $ (cf. \cite[2.4.13]{Corti07}).
It follows that  
$$a_*M\le a _*  (\lceil \frac ji\mathbf M _{i, \bar S}+\mathbf A _{\bar S}\rceil )= a _* \mathbf M_{j,\bar S}\le a_*(j\mathbf{D}_{\bar S}),$$
where the second inequality follows from \eqref{l-ex}.
Thus $a _* \mathbf D_{\bar S}$ is rational and hence we may have that $a _* (j\mathbf D_{\bar S})=a _* \mathbf M_{j,\bar S}$.
Since $\mathbf D_{\bar S}$ descend to $S^+$, we have that $\mathbf D_{\bar S}$ is rational.
\end{proof}


\subsection{Existence of pl-flips}\label{ss-Qlimit}
Up to this point, our arguments apply to any 3 dimensional pl-flip. However, in this subsection, we will require that the characteristic of the ground field is larger than 5 and the coefficients are in the standard set $\{ \frac{n-1}{n}|n\in \mathbb{N}\}$. 
\begin{theorem}\label{t-2}Let $f:(X,S+B)\to Z$ be a pl-flipping contraction of a projective threefold defined over an algebraically closed field of characteristic $p>5$ such that the coefficients of $B$ belong to the standard set $\{ 1-\frac 1 n|n\in \mathbb N\}$. Then the flip exists. 
\end{theorem}
\begin{proof}
Since, by adjunction, 
$$K_{S^n}+B_{S^n}=(K_X+S+B)|_{S^n},$$
we have that the coefficients of $(S^n,B_{S^n})$ also lie in $\{1-\frac 1 n |n\in \mathbb N\}$ (cf. \cite[3.36]{Kollar13}) and so, by \eqref{t-gfr}, $(S^n,B_{S^n})$ is strongly $F$-regular over $Z$.  By \eqref{p-normal}, we have that $S$ is normal. 

By \cite{Corti07}, it suffices to show that the restricted algebra
$R_{S/Z}(k(K_X+S+B))$ is finitely generated for some $k>0$.
Recall that the restricted algebra is a graded $\mathcal O _Z$-algebra  whose degree $m$ piece corresponds to the image of the restriction homomorphism $$f_*\mathcal O _X(mk(K_X+S+B))\to f_*\mathcal O _{S}(mk(K_{S}+B_{S})),$$
or equivalently to  the image of the restriction homomorphism $$(f\circ g)_*\mathcal O _Y(mk(K_Y+S'+B_{Y}))\to (f\circ g)_*\mathcal O _{S'}(mk(K_{S'}+B_{S'})),$$ where $B_{Y}=(-\mathbf A_{Y})_{\geq 0}$ and $B_{S'}=(-\mathbf A_{S'})_{\geq 0}$.
Recall that $Q\sim k(K_X+S+B)$ for $k>0$ sufficiently divisible. Replacing $k$ by a multiple, we may assume that $Q$ is Cartier and that $a_* \mathbf D_{\bar S}=a_*\mathbf D_{j,\bar S}$ for all $j>0$ by \eqref{c-rat}.

Pick a rational number $\delta>0$ such that 
\begin{enumerate}
\item  $(X,S+B+\delta A^*)$ is plt,
\item $\nu:\bar{S}\to S$ is also a terminalization of $(S,B_{S}+\delta A^*|_{S})$,
\end{enumerate}
\begin{lemma}\label{l-liftingsurface}
Replacing $j$ by a multiple, for any $i\gg 0$ divisible by $j$, we have 
$$S^0(S',\sigma(S',\Psi_{S'})\otimes \mathcal O_{S'}(L_{ij}|_{S'}))=H^0(S^+, \mathcal O_{S^+}( j\mathbf{D}_{S^+})).$$ 
\end{lemma}
\begin{proof} Possibly replacing $j$ by its multiple, we can assume $|j\mathbf{D}_{\bar{S}}|$ induces a morphism to a normal surface $S^+$. 

Let $G_\epsilon=(\Psi-\Psi'_\epsilon +\epsilon F)|_{S'}$ (see \eqref{l-psi} for the definitions of the divisors) and 
$$\Psi _{\bar{S}}:=\mu_*\Psi_{S'}=B_{\bar{S}}+\mu_*G_\epsilon+j(\mathbf{D}_{\bar{S}}-\mathbf{D}_{i,\bar{S}}) $$ for sufficiently large $i$ and $\Psi_{S^+}=a_*\Psi_{\bar{S}}$. (We have used the fact that $B_{\bar S}=\{ B_{\bar S}\}=\{-\mathbf A _{\bar S}\}$, that $j\mathbf D _{S'}=\mu ^*(j\mathbf D_{\bar S})$ is integral, and that $\{ - \frac ji\mathbf M _{i,S'}\} =j\mathbf{D}_{{S'}} -\frac ji\mathbf M _{i,S'}$ for $i\gg 0$.)

Since $(S,B_{S})$ is globally $F$-regular (cf. \eqref{t-gfr}), so is  $(\bar{S},B_{\bar{S}})$ (cf. \eqref{p-bir}). Because $0< \epsilon \ll 1$, $||\Psi-\Psi'_\epsilon ||\ll 1$, and $i\gg 0$, by  \eqref{l-GFR}, it follows that $(\bar{S},\Psi _{\bar{S}})$ is globally $F$-regular,  and so
$(S^+,\Psi  _{S^+})$ has strongly $F$-regular singularities (see  \eqref{l-gtos}).  Let $E$ be an effective $\mathbb{Q}$-divisor on $S^+$ whose support contains the image of the locus where $S'\to S^+$ is not an isomorphism (this is possible as we have assumed that $Y\to Y_0$ is an isomorphism over $X\setminus {\rm Ex}(f)$), the divisor defined in \eqref{l-perturbation}, and the birational transform of ${\rm Supp}(B_S+A^*|_{S})$.
We may assume that\begin{enumerate}
\item  $(\bar S, B_{\bar S}^\sharp=B_{\bar S}+2a^*E)$ is globally $F$-regular over $S^+$ and 
\item $G_\epsilon +j(\mathbf D_{ S'}-\mathbf D_{i, S'})\leq \mu ^*a^*E$ (for fixed $j$ and $i\gg 0$).\end{enumerate}

To see (2), note that the support of $G_\epsilon +j(\mathbf D_{ S'}-\mathbf D_{i, S'})$ is contained in the support of $\mu ^*a^*E$ and we have that $||j(\mathbf D_{S'}-\mathbf D_{i, S'})||\ll 1$ for $i\gg 0$. From our choice of $\epsilon$ and $F$, $A^*$ we can also assume that $\epsilon F|_{S'}\leq \frac{1}{3} \mu ^* a^* E$ (see (\ref{l-prel}.4)) and 
 $(\Psi -\Psi'_\epsilon)|_{S'}\le g^*A^*|_{S'}\le \frac{1}{3}\mu^*a^*E  $.  

\begin{claim} We have the following inclusion
$$S^0(S',\sigma(S',\Psi_{S'})\otimes \mathcal{O}_{S'}(L_{ij}|_{S'}))\supset S^0(\bar{S}, \sigma(\bar{S},B_{\bar{S}}^\sharp )\otimes \mathcal{O}_{\bar{S}}(j\mathbf{D}_{\bar{S}})).$$
\end{claim}
\begin{proof}
Since $\mu _* \Psi _{S'}=\Psi_{\bar S}$ and $\mu _* (L_{ij}|_{S'})=j\mathbf D _{\bar S}$, there is a commutative diagram
{\small
\begin{diagram}
\mu_*F^e_*\mathcal{O}_{S'}((1-p^e)(K_{S'}+\Psi _{S'})+p^eL_{ij}|_{S'})&\rTo &\mu_*\mathcal{O}_{S'}(L_{ij}|_{S'})\\
\dTo& & \dTo  \\
F^e_*\mathcal{O}_{\bar{S}}((1-p^e)(K_{\bar{S}} +\Psi _{\bar{S}})+p^ej\mathbf{D}_{\bar{S}})&\rTo & \mathcal{O}_{\bar{S}}(j\mathbf{D}_{\bar{S}}).
\end{diagram}
}
As in the argument of \eqref{p-stab} for $e\gg 0$, the image of the map on global sections induced by the top arrow is $S^0(S',\sigma(S',\Psi_{S'})\otimes \mathcal O_{S'}(L_{ij}|_{S'}))$, thus it suffices to show that for $i\gg 0$, the top left hand corner contains 
{
$$F^e_*\mathcal{O}_{\bar{S}}((1-p^e)(K_{\bar{S}} +B_{\bar{S}}^\sharp )+p^e(j\mathbf{D}_{\bar{S}})).$$
 }
Let
$$\Psi _{\bar S}^*:= B_{\bar{S}}+\mu_*G_\epsilon+j(\mathbf{D}_{\bar{S}}-\mathbf{D}_{i,\bar{S}})+a^*E\leq B_{\bar S}^\sharp $$ 
(cf. (2) above).
It suffices to show that
 {\small
 \begin{eqnarray*}
& (1-p^e)(K_{S'}+\Psi _{S'})+p^eL_{ij}|_{S'}- \mu^*((1-p^e)(K_{\bar{S}}+B_{\bar S}^\sharp  )+p^e(j\mathbf{D}_{\bar{S}}))\\
 \ge &(1-p^e)(K_{S'}+\Psi _{S'})+p^eL_{ij}|_{S'}- \mu^*((1-p^e)(K_{\bar{S}}+\Psi ^{*}_{\bar{S}})+p^e(j\mathbf{D}_{\bar{S}}))\\
 = &(p^e-1)\left(\lceil \frac{j}{i}\mathbf{M}_{i,S'}+\mathbf{A}_{S'}\rceil -K_{S'}-\Psi _{S'}-\mu^*(j\mathbf{D}_{\bar{S}}-K_{\bar{S}}-\Psi ^{*}_{\bar{S}})\right)+L_{ij}|_{S'} - \mu^*(j\mathbf{D}_{\bar{S}}) \\
 =&(p^e-1)\left(\frac{j}{i}\mathbf{M}_{i,S'}+\mathbf{A}_{S'}- G_\epsilon-K_{S'}-\mu^*(j\mathbf{D}_{\bar{S}}-K_{\bar{S}}-\Psi_{\bar S}-a^*E )\right)+L_{ij}|_{S'} - \mu^*(j\mathbf{D}_{\bar{S}})\\
 \ge&(p^e-1)(   \mu^* a^*E- G_\epsilon+\frac{j}{i}\mathbf{M}_{i,S'}-j\mu ^*\mathbf{D}_{\bar S})+L_{ij}|_{S'} - \mu^*(j\mathbf{D}_{\bar{S}})
 \end{eqnarray*}
}
 is effective, where the last inequality follows since 
 $$\mu ^*(K_{\bar S}+\Psi_{\bar S})-K_{S'}+\mathbf A  _{S'}=\mu^*\mu_*G_\epsilon +j(\mathbf{D}_{S'}-\mathbf{D}_{i,S'})\ge 0.$$
 Note that for any fixed $\delta>0$ and $i\gg 0$ we have $$\frac{j}{i}\mathbf{M}_{i,\bar{S}}\geq j \mathbf{D}_{\bar{S}}-\delta \nu^*(A^*|_S)$$ and so
$$\frac{j}{i}\mathbf{M}_{i,{S'}}+\mathbf A_{S'}=\frac{j}{i}\mu ^*\mathbf{M}_{i,\bar{S}}+\mathbf A_{S'}\geq 
\mu ^*(j \mathbf{D}_{\bar{S}}-\delta \nu^*(A^*|_S))+\mathbf A _{S'}.$$
Thus, as $j\mathbf{D}_{\bar{S}}$ is Cartier by \eqref{c-rat},  we have that
$$\lceil \frac{j}{i}\mathbf{M}_{i,{S'}}+\mathbf A_{S'}\rceil\geq \mu ^* (j\mathbf{D}_{\bar{S}})+\lceil -\delta  g^* (A^*|_{S'})+\mathbf A _{S'}\rceil \geq \mu^*(j\mathbf{D}_{
\bar{S}}),$$
where for the last inequality we have used the fact that since $(S,B_{S}+\delta A^*|_{S})$ is klt,
we have that $\lceil -\delta g^*(A^*|_S)+\mathbf A_{S'}\rceil\geq 0$.
Therefore, we have that $L_{ij}|_{S'}-  \mu^*(j\mathbf{D}_{\bar{S}})\geq 0$ is an effective divisor.

Finally, note also that as $\mu ^* \mu _*G_\epsilon \geq G_\epsilon$, it follows from (2) that 
 $$  \mu^* a^*E-  G_\epsilon+\frac{j}{i}\mathbf{M}_{i,{S'}}-j\mathbf{D}_{S'}\geq 0$$  and so the required inequality follows. \end{proof}

By \eqref{l-rcu3}, we have that 
{ $$ S^0(\bar{S}, \sigma(\bar{S},B_{\bar S}^\sharp )\otimes \mathcal{O}_{\bar S}(j\mathbf{D}_{\bar{S}}))\supset H^0(S^+, \mathcal O_{S^+}( j\mathbf{D}_{S^+})).$$}
Thus we have shown that
$$S^0(S',\sigma(S',\Psi_{S'})\otimes \mathcal O_{S'}(L_{ij}|_{S'}))\supset H^0(S^+, \mathcal O_{S^+}( j\mathbf{D}_{S^+})).$$ The reverse inclusion is clear as $a_*\mu _* (L_{ij}|_{S'})=j\mathbf D _{S^+}$. Thus the lemma follows.
\end{proof} 

By construction, we have that 
$$R_{S/Z}(k(K_X+S+B))\subset R({S^+/Z};j\mathbf{D}_{S^+}).$$
Since $|L_{ij}| \subset |j k(K_Y+S'+B_{Y})|$ we have that $$S^0(Y, \sigma(Y,\Xi)\otimes \mathcal O_Y(L_{ij}))\subset H^0(Y,\mathcal O _Y(jk(K_Y+S'+B_{Y}))).$$ The above lemma together with \eqref{YtoS} imply that 
$$R_{S/Z}(jk(K_X+S+B))\supset R({S^+/Z};j\mathbf{D}_{S^+})$$ for $j>0$ sufficiently divisible and thus equality holds. Since  
$$R({S^+/Z};j\mathbf{D}_{S^+})\cong R({\bar S/Z};j\mathbf D _{\bar S})$$ 
is a finitely generated $\mathcal O _Z$-algebra (cf. \eqref{l-sab}), so is $R_{S/Z}(k(K_X+S+B))$  and the theorem holds.
\end{proof}

\section{On the minimal model program for 3-folds}\label{s-mmp}
\subsection{The results of Keel}
\begin{thm}[{\cite[0.2]{Keel99}}]\label{t-keel1} Let $L$ be a nef line bundle on a scheme $X$, projective over an algebraically closed field of characteristic $p>0$. $L$ is semi-ample if and only if $L|_{\mathbb E (L)}$ is semi-ample. 
In particular, if the basefield is the algebraic closure of a finite field and $L|_{\mathbb E (L)}$ is numerically trivial, then $L|_{\mathbb E (L)}$ is semi-ample.
\end{thm}
Recall that for a nef line bundle $L$, $\mathbb E (L)$ is the closure of the union of all of those irreducible subvarieties with $L^{\dim Z}\cdot Z=0$.
\begin{thm}[{\cite[0.5]{Keel99}}]\label{t-keel2} Let $X$ be a normal $\Q$-factorial three-fold, projective over an algebraically cosed field of positive characteristic. Let $L$ be a nef and big line bundle on $X$.
If $L-(K_X +\Delta)$ is nef and big for some boundary $\Delta$ with $\lfloor \Delta\rfloor =0$, then $L$ is EWM. If the basefield is the algebraic closure of a finite field, then $L$ is semi-ample.\end{thm}
Recall that a nef line bundle $L$ on a scheme $X$ proper over a field is {\it Endowed With a Map (EWM)} if there is a proper map $f:X\to Y$ to a proper algebraic space which contracts exactly ${\mathbb E (L)}$.
\begin{thm}[{\cite[0.6]{Keel99}}]\label{t-keel3} 
Let $X$ be a normal $\Q$-factorial three-fold, projective over an algebraically closed field. Let $\Delta$ be a boundary on X. If $K_X +\Delta$ has nonnegative Kodaira dimension, then there is a countable collection of curves $\{ C_i \}$ such that	
\begin{enumerate}\item $\overline{NE}_1(X) = \overline{NE}_1(X) \cap (K_X + \Delta)_{\geq 0} +	\mathbb R _{\ge 0}\cdot [C_i]$. 
\item All but a finite number of the $C_i$ are rational and satisfy $ 0 < −(K_X + \Delta) \cdot C_i \leq 3$.
\item The collection of rays $\{ R \cdot [C_i]\}$ does not accumulate in the half-space $(K_X+\Delta )_{<0}$.
\end{enumerate}\end{thm}
We have the following easy consequence of \eqref{t-keel2}.
\begin{thm}\label{t-cont} Let $(X,\Delta )$ be a normal $\Q$-factorial three-fold dlt pair projective over a field of positive characteristic $p>0$ such that $K_X+\Delta $ is pseudo-effective.  Let $R$ be a $(K_X+\Delta)$-negative extremal ray
\begin{enumerate}
\item Then the corresponding contraction $f:X\to Z$ exists in the category of algebraic spaces. 
\item If $k=\overline{\mathbb{F}}_p$ is the algebraic closure of a finite field, then $f:X\to Z$ is a projective morphism.
\item If  $(X,\Delta=S+B)$ is plt, $S$ is normal and $S\cdot R<0$ then $f:X\to Z$ is a projective morphism with $\rho(X/Z)=1$. 
\end{enumerate}
\end{thm}
\begin{proof} Suppose that $K_X+\Delta$ is not nef, then by \eqref{t-keel3}, it follows easily that there is an ample $\mathbb{Q}$-divisor $H$ such that $H+K_X+\Delta$ is nef and 
$$\overline{NE}_1(X)\cap  (H+K_X+\Delta)^{\perp}=\mathbb{R}_{\ge 0} [R] .$$ Since $K_X+\Delta $ is pseudo-effective, we have that $L:=H+K_X+\Delta$ is big. Then \eqref{t-keel2} immediately implies (1) and (2). 

(3) follows from \eqref{t-keel1}. Note that $\mathbb{E}(L)\subset S$ and $L|_{S}=K_S+{\rm Diff }_S(B+H)$ is semiample by the contraction theorem in the surface case (see e.g. \cite[2.3.5]{KK}). Using \eqref{t-keel1}, then the usual argument implies $\rho(X/Z)=1$ (see \cite[3.17]{KM98}). 
\end{proof}
\subsection{Existence of flips and minimal models}
\begin{defn} Let $(X,\Delta)$ be a dlt pair such that $X$ is $\mathbb{Q}$-factorial. We say that $f:X\to Z$ an {\it  extremal flipping contraction} if it is a projective birational morphism between normal quasi-projective varieties such that
\begin{enumerate}
\item $f$ is small (i.e. an isomorphism in codimension 1);
\item $-(K_X+\Delta)$ is ample over $Z$;
\item $\rho (X/Z)=1$.
\end{enumerate}
\end{defn}

\begin{proof}[Proof of \eqref{t-flipexists}] Replacing $\Delta$ by $\Delta -\frac 1 n\lfloor \Delta \rfloor$ for some $n\gg 0$, we may assume that $(X,\Delta )$ is klt. We use Shokurov's reduction to pl-flips and special termination as explained in \cite{Fujino07}.
We simply indicate the changes necessary to the part of the argument that is not characteristic independent.

Step 2 of \cite[4.3.7]{Fujino07} requires resolution of singularities. We use \eqref{t-res}.

Step 3 of \cite[4.3.7]{Fujino07} requires that we run a relative minimal model program. Since all divisors are relatively big in this context, by \eqref{t-keel3} the relevant negative extremal ray $R$ always exists. 
We may assume that all induced contractions are $K_Y+S+B$ negative for an appropriate plt pair $(Y,S+B)$ where $S\cdot R<0$ (if the pair $(X,S+B)$ is dlt, then we may replace $B$ by $B-\frac 1 n \lfloor B \rfloor$ for $n\gg 0$).
Thus
the corresponding contraction morphism exists and is projective and all required divisorial contractions exist. Since all flipping contractions are pl-flipping contractions the corresponding flips  exist by \eqref{t-2}.

By Step 5 of \cite[4.3.7]{Fujino07}, we obtain $f^+:(X^+,\Delta ^+)\to Z$ a small birational morphism such that $X^+$ is $\Q$-factorial, $(X^+,\Delta ^+)$ is dlt and $K_{X^+}+\Delta ^+$ is nef over $Z$.
However, it is clear that $\Delta ^+$ is the strict transform of $\Delta$ and hence $\lfloor \Delta ^+\rfloor =0$ so that $(X^+,\Delta ^+)$ is klt. Denote by $p$ the number of exceptional divisors of a common resolution $Y\to X$. Since $X$ and $X^+$ are $\mathbb{Q}$-factorial, we have 
$$\rho(X)=\rho(Y)-p=\rho(X^+),$$
which implies  $\rho (X^+/Z)=\rho (X/Z)=1$ and hence that $K_{X^+}+\Delta ^+$ is ample over $Z$. Thus $X\dasharrow X^+$ is the required flip. 
\end{proof}
\begin{thm}\label{t-quasi} Let $k$ be an algebraically closed field of characteristic $p>5$, $(X,\Delta )$ a projective $\Q$-factorial $3$-fold klt pair such that $K_X+\Delta $ is pseudo-effective and all coefficients of $ \Delta $ are in the standard set $\{ 1-\frac 1 n | n\in \mathbb N\}$. Let $R$ be a $K_X+\Delta$ negative extremal ray and  $f:X\to Z$ the corresponding proper birational contraction to a proper algebraic space (given by \eqref{t-cont}) such that a curve $C$ is contracted if and only if $[C]\in R$.
\begin{enumerate}
\item There is a small birational morphism $f^+: X^+\to Z$ such that $X^+$ is $\Q$-factorial and projective, and $K_{X^+}+\Delta^+$ is nef over $Z$ where $\Delta ^+$ denotes the strict transform of $\Delta$.
\item If moreover, $f$ is divisorial, then $X^+=Z$ and in particular $Z$ is projective.
\end{enumerate}
\end{thm}
\begin{proof} (1) The extremal ray $R$ is cut out by a big and nef $\mathbb Q$-divisor of  the form $L=K_X+\Delta+H$ for some ample $\Q$-divisor $H$. By \eqref{t-keel3} $L$ is EWM. Let $f:X\to Z$ be the corresponding  birational contraction to a proper algebraic space.
Since $K_X+\Delta+(1-\epsilon)H$ is big for any rational number $0<\epsilon \ll 1$, there is a positive (sufficiently divisible) integer $m\in \mathbb{N}$ and a divisor $S\sim m(K_X+\Delta+(1-\epsilon)H)$. Thus $S\cdot R <0$. Replacing $S$ by a prime component, we find a prime divisor $T$ with $T\cdot R<0$. In the divisorial contraction case, $T$ is the contracted divisor. 

Consider a log resolution $\nu: Y\to X$ of the pair $(X,\Delta+T)$, and let $E$ be the reduced $\nu$-exceptional divisor. Write  
$$\Delta=\sum_{\Delta_i\neq T}a_i\Delta_i+bT=\Gamma+bT.$$
We run the $(K_Y+\Gamma'+E+T')$-MMP over $Z$, where  $\Gamma'=\nu ^{-1}_*\Gamma$ and $T'=\nu ^{-1}_*T$ are birational transforms of $\Gamma$ and $T$. Note that running the MMP over $Z$ means that at each step we only consider extremal curves $C^j$ such that $C^j\cdot L^j=0$, where $L^j$ is the strict transform of $\nu ^*L$. We obtain models
$$Y=Y^1\dasharrow   Y^2\dasharrow \ldots$$
We note that since $K_X+\Delta +H= L\equiv _Z 0$ it follows easily that $K_X+\Delta \equiv _Z aT$ for some $a>0$ and so 
$$K_Y+\Gamma'+T'+E\equiv_{\mathbb{Q}} f^*(K_X+\Delta)+\sum c_iE_i+(1-b)T'\equiv_{Z} \sum d_iE_i+cT',$$
where $d_i\ge c_i>0$ and $c>1-b>0$.

 For each step of this MMP, we let $C^j$ be a curve spanning the corresponding extremal ray and $\bullet^j$ is the push forward of the divisor $\bullet$ from $Y$ to $Y^j$. We have
$$ C^j\cdot (\sum d_iE^j_i+cT^j)<0,$$ and so at each step of the MMP, the curve $C_j$ intersects a component of $T^j+E^j$ negatively. By special termination, this MMP terminates after finitely many steps and we obtain a minimal model over $Z$, say $W=Y^m$.

Next, we run the $(K_{W}+\Gamma_W+E_W+bT_W)$-MMP with scaling by $T_W$ which yelds birational maps
$$W=W^1\dasharrow W^2\dasharrow\ldots $$  It is easy to see that each step is $(\sum d_iE^k_i)$-negative, and hence also of special type. By special termination, this MMP terminates after finitely many steps and we obtain a minimal model over $Z$, say $X^{+}=W^l$.

Let $p:U\to X$ and $q:U\to X^+$ be a common resolution and consider the divisor 
$p^*(K_X+\Delta )-q^*(K_{X^+}+\Gamma_{X^+}+bT_{X^+}+E_{X^+})$ which is easily seen to be anti-nef and exceptional over $X$.
By the negativity lemma, this divisor is effective and so $X\dasharrow X^+$ is a birational contraction and $E_{X^+}=0$. It follows that $X\dasharrow X^+$ is a minimal model for $K_X+\Delta$ over $Z$. 

\vspace{3mm}

\noindent (2) It remains to show that if $f:X\to Z$ is divisorial contraction, then $Z\cong X^+$.
Assume that $f^+:X^+\to Z$ is not an isomorphism. Let $C$ be a $f^+$-contracted curve, and $H^+$ an ample divisor on $X^+$, so $C\cdot H^+>0$. Denote by $H$ the birational transform of $H^+$ on $X$. Let $R$ be the exceptional ray and $E$ be an $f$-exceptional divisor.  Since, $E\cdot R<0$, there exists a number $a\in \mathbb Q$ such that $(H+aE)\cdot R=0$.

Let  $p:U\to X$ and $q:U\to X^+$  be a common resolution, then $p^*(H+aE)-q^*(H^+)\equiv _{X^+}0$ is $q$-exceptional and so $p^*(H+aE)=q^*H^+$ (by the negativity lemma). But then $q^*H^+\equiv p^*(H+aE) \equiv _Z 0$ contradicting the fact that $C\cdot H^+>0$. 
 \end{proof}
  
 In \eqref{t-quasi}, if $f$ is small then we will say that $X^+\to Z$ is the corresponding {\it generalized flip}. We say that $X=X_1\dasharrow X_2\dasharrow \ldots $ is {\it a  generalized $(K_X+\Delta)$-MMP}, if each map $X_i\dasharrow X_{i+1}$ is a  $(K_X+\Delta)$-divisorial contraction or a $(K_X+\Delta)$-generalized flip. 
 
\begin{proof}[Proof of \eqref{t-MM}] When $k=\overline{\mathbb{F}}_p$, we can run the usual $(K_X+\Delta)$-MMP by \eqref{t-cont} and \eqref{t-flipexists}; and in general we can run the generalized $(K_X+\Delta)$-MMP defined as above by \eqref{t-quasi}. The condition $N_\sigma (K_X+\Delta )\wedge \Delta=0$ guarantees that no component of $\Delta$ is contracted and hence that $(X,\Delta )$ remains canonical at each step.

So it suffices to show that a sequence of the generalized $(K_X+\Delta)$-minimal model program terminates. 
In fact this directly follows from the argument of \cite[6.17]{KM98}. We define the non-negative integer valued function $d(X,\Delta)$ as in \cite[6.20]{KM98}. We easily see $d(X,\Delta)<d(X^+,\Delta^+)$ as long as ${\rm Supp (\Delta^{+})}$ contains an exceptional curve of $X^{+}\to Z$. Thus for a sequence of flips, the birational transform of $\Delta$ eventually will not contain flipping curves. 
Then the rest of the argument is precisely the same as the one in \cite[6.17]{KM98}.
\end{proof}

\section{Applications to 3-fold singularities}
In this section, we present applications of our results to the study of singularities. In characteristic 0, all these applications are known to follow from  the minimal model program. In our context, some difficulties arise since we have restrictions on the coefficients of the boundaries.
\begin{theorem}Fix $k$ is an algebraically closed field of characteristic $p>5$, $(X,B)$ a $3$-fold  pair such that $B\in \{ 1-\frac 1 n |n\in \mathbb N \cup \infty \}$. Then \begin{enumerate}
\item If $(X,B)$ is canonical, then it has a $\mathbb Q$-factorial terminalization i.e. a proper birational morphism $\bar f:\bar X\to X$ such that $\bar X$ is  $\mathbb Q$-factorial and the exceptional divisors are the set of divisors $E$ over $X$ with $a_E(X,B)=0$.
\item  $(X,B)$ has a dlt modification i.e. a proper birational morphism $\bar f:\bar X\to X$ such that 
$\bar X$ is $\mathbb Q$-factorial, the exceptional divisors are the set of divisors $E$ over $X$ with $a_E(X,B)=-1$, if $\bar B=\bar f ^{-1}_*B+\bar E$ where $\bar E$ is the reduced exceptional divisor, then $(\bar X,\bar B)$ is dlt and $K_{\bar X}+\bar B$ is semiample over $X$, i.e., the log canonical modification exists and is given by $X^{\rm LC}={\rm Proj}R( K_{\bar X}+\bar B/X)$.\end{enumerate}
\end{theorem}
\begin{proof} (1) Let $\tilde f:\tilde X\to X$ be a log resolution, write $K_{\tilde X}+\tilde B=\tilde f^*(K_X+B)+\tilde F$ where $\tilde B$ is the strict transform of $B$. Since $K_{X}+B$ is canonical, we have $\tilde F\geq 0$ and $\tilde B\wedge \tilde F=0$. We run the mmp over $Y$. Since each step is $\tilde F$-negative,  it is special MMP. Therefore, by special termination (cf. \cite[4.2]{Fujino07}), this mmp is eventually disjoint from  $\tilde F$ and thus it terminates. The outcome is a pair $(\bar X,\bar B)$ with a morphism $\bar f:\bar X\to X$ such that $K_{\bar X}+\bar B$ is nef over $X$. By the negativity lemma $\bar F\leq 0$ and hence $\bar F=0$. It is easy to see that $\tilde X \dasharrow \bar X$ 
only contracts divisors in $\tilde F$ and hence divisors with discrepancy $>0$.

(2) Let $\tilde f:\tilde X\to X$ be a log resolution, write $K_{\tilde X}+\tilde B=\tilde f^*(K_X+B)+\tilde F$ where $\tilde B$ is the strict transform of $B$ plus the reduced exceptional divisor $\tilde E$. We run the mmp for $K_{\tilde X}+\tilde B$ over $X$.  By special termination (cf. \cite[4.2]{Fujino07}), this is eventually disjoint from the exceptional divisor. Since the support of $\tilde F$ is contained in the support of $\tilde E$, this mmp terminates with $\phi :\tilde X \dasharrow \bar X$ and a morphism $\bar f:\bar X\to X$. Note that $(\bar X , \bar B)$ is dlt.

Since $K_{\bar X}+\bar B$ is nef and $\bar F$ is exceptional, by the negativity lemma we have that $\bar F\leq 0$. Note that $\bar f({\rm Supp} (- \bar F))$ is the non log canonical locus. Since $\bar{F}\le 0$ and it is nef over $X$, we know that $\bar f^{-1}(\bar f({\rm Supp}(-\bar{F})))={\rm Supp}(-\bar{F})$.

We write ${\rm Ex}(\bar{f})=W_1\cup W_2$, where $W_i$ is the union of the $i$-dimensional components of the exceptional locus with reduced structure. In particular, $\bar{f}(W_1)$ consists of isolated points and $W_1\cap {\rm Supp}(-\bar{F})=\emptyset.$

We know that $(K_{\bar{X}}+\bar{B})|_{W_1}=\bar f^*(K_X+B)|_{W_1}$ is trivial on $W_1$. By 2-dimensional abundance for sdlt pairs, we have that $(K_{\bar X}+\bar B)|_{W_2}$ is semiample over $X$. Therefore, $K_{\bar{X}}+\bar{B}$ is semi ample on $W_1\cup W_2$ (see \cite[2.12]{Keel99}). Hence by Keel's result (\cite[0.2]{Keel99}), we have that $K_{\bar X}+\bar B$ is semiample over $X$.


\end{proof}
\begin{theorem}Fix $k$ an algebraically closed field of characteristic $p>5$, $(X,S+B)$ a $3$-fold  pair such that $S$ is a prime divisor and $B\in \{ 1-\frac 1 n |n\in \mathbb N  \cup \infty \}$. Then \begin{enumerate}
\item $(X,S+B)$ is lc on a neighborhood of $S$ iff $(S^n,B_{S^n})$ is lc where $\nu :S^n\to S$ is the normalization and $K_{S^n}+B_{S^n}=\nu ^*(K_X+S+B)$,
\item  If $X$ is $\mathbb{Q}$-factorial, then $(X,S+B)$ is plt on a neighborhood of $S$ iff $(S^n,B_{S^n})$ is klt.
 \end{enumerate}\end{theorem}
\begin{proof} (1) We must show that if $(S^n,B_{S^n})$ is lc then $(X,S+B)$ is lc on a neighborhood of $S$ (the reverse implication is trivial). Let $\bar f:\bar X \to X$ be the dlt model, then $K_{\bar X}+\bar B=\bar f^*(K_X+B)+\bar F$ is nef over $X$ and $\bar F\leq 0$. Thus every fiber is either contained in the support of $\bar F$ or disjoint from it. The non-lc locus of $(X,B)$ is the image of ${\rm Supp}(-\bar F)$. Suppose by contradiction that $\bar f ({\rm Supp}(-\bar F))$ intersects $S$ at a point $x$, then $\bar S =f^{-1}_*S$ intersects ${\rm Supp}(-\bar F)$. Let $P$ be a codimension $1$ point contained in this intersection, then $a_P(S^n,B_{S^n})<-1$ which is a contradiction.

(2) We must show that if $(S^n,B_{S^n})$ is klt then $(X,S+B)$ is plt on a neighborhood of $S$ (the reverse implication is trivial). Let $\bar f:\bar X \to X$ be the dlt modification, then $K_{\bar X}+\bar B=\bar f^*(K_X+B)$ is nef over $X$. Let $\bar E$ be the reduced exceptional divisor.  By restricting to a neighborhood of $S$, we assume the image of each component $\bar E$ intersects $S$.  Since $X$ is $\mathbb{Q}$-factorial, $\bar E$ is the exceptional locus. Assume $\bar E\neq 0$, then  we conclude that $\bar E$ intersects the birational transform of $S$ on $\bar{X}$ and hence $(S^n,B_{S^n})$ is not klt. This is a contradiction. Thus  $\bar E=0$ and hence $(X,S+B)$ is plt on a neighborhood of $S$. \end{proof}

\begin{theorem}Fix $k$ is an algebraically closed field of characteristic $p>5$ and $(X,B)$ a $3$-fold  pair. If $f:X\to C$ is a flat family over a smooth curve such that for any $c\in C$ the pair $(X,B+X_c)$ is log canonical, then
for any finite map of smooth curves $\tilde C \to C$ and $\tilde X =X\times _C \tilde C\to \tilde C$ we have that $(\tilde X,\tilde B+\tilde X_{\tilde c})$ is log canonical for any $\tilde c\in \tilde C$ where $K_{\tilde X}+\tilde B=f^*(K_X+B)$.\end{theorem}
\begin{proof} It is easy to see that $B$ contains no fiber components and so the coefficients of $\tilde B$ are the same as those of $B$. By adjunction $K_{X_c}+B_c=(K_X+X_c+B)|_{X_c}$ is lc and hence so is 
$K_{\tilde X_{\tilde c}}+\tilde B_{\tilde c}=(K_{\tilde X}+\tilde X_{\tilde c}+\tilde B)|_{\tilde X_{\tilde c}}$ (because $\tilde X_{\tilde c}\to X_c$ is an isomorphism).
By inversion of adjunction $K_{\tilde X}+\tilde X_{\tilde c}+\tilde B$ is lc near $\tilde X_{\tilde c}$.
\end{proof}
Recall now that 
a point on a $3$-fold  $P\in X$
is a cDV (i.e. a compound Du Val) point if a general hyperplane section has at worst a Du Val singularity at
$P$. We have the following.
\begin{thm} Let 
a point on a $3$-fold  $P\in X$
be cDV over an algebraically closed field of characteristic $>5$, then for a general hyperplane section $H$, $(X,H)$ is canonical (in a neighborhood of $P$). 
\end{thm}
\begin{proof}Since $H$ has embedded dimension at most 3, we know that $P\in X$ has embedded dimension at most 4. So $P\in X$ is a hyperplane singularity. In particular, $P\in X$ is Gorenstein. Let $\tilde f:\tilde X\to X$ be a log resolution, write $K_{\tilde X}+\tilde H=\tilde f^*(K_X+H)+\tilde F-\tilde E$ where  $H$ is a general hyperplane through $P$ and $\tilde H$ is the strict transform of $H$. Assume $\tilde E \neq 0$. We run the $K_{\tilde X}+\tilde H$ mmp over $X$.
Suppose that $X_i\to X_{i+1}$ is a divisorial contraction of a component $E_{i,j}$ of $E_i$ the strict transform of $\tilde E =\sum e_j \tilde E_j$. For a general contracted curve $\Sigma _i\subset E_{i,j}$, we have 
$\Sigma _i\cdot F_i\geq 0$ and $\Sigma _i\cdot  E_{i,j}<0$. Since $(F_i-E_i)\cdot \Sigma _i=(K_{X_i}+H_i)\cdot \Sigma _i<0$, it follows that $E_i$ is not irreducible.
Therefore, if $X_N$ is the minimal model over $X$, we have that $E_N\ne 0$. On the other hand, by the negativity lemma, we have $F_N=0$. Since $-E_N$ is nef, it contains any fiber that it intersects and so $H_N\cap E_N\ne \emptyset$. It then follows that $H$ is not canonical which is a contradiction.

\end{proof}


\enddocument